\documentclass[12pt]{amsart}
\usepackage{amsmath}
\usepackage{amstext}
\usepackage{amsfonts}
\usepackage{amssymb}
\usepackage{amsthm}
\usepackage{amsrefs}
\usepackage{bbm}

\usepackage{microtype}
\usepackage[colorlinks=true, allcolors=blue]{hyperref}

\theoremstyle{plain}
\newtheorem{thm}{Theorem}[section]

\newtheorem{cor}[thm]{Corollary}
\newtheorem{prop}[thm]{Proposition}

\theoremstyle{definition}
\newtheorem{defn}[thm]{Definition}
\newtheorem{exam}[thm]{Example}
\newtheorem{rem}[thm]{Remark}

\newtheorem{problem}[thm]{Problem}

\newcommand{\bC}{{\mathbb{C}}}

\newcommand{\bN}{{\mathbb{N}}}
\newcommand{\bR}{{\mathbb{R}}}
\newcommand{\bT}{{\mathbb{T}}}

\newcommand{\B}{{\mathcal{B}}}

\newcommand{\M}{{\mathcal{M}}}

\newcommand{\rA}{{\mathrm{A}}}
\newcommand{\rB}{{\mathrm{B}}}
\newcommand{\rC}{{\mathrm{C}}}

\renewcommand{\phi}{\varphi}

\newcommand{\id}{\operatorname{id}}
\newcommand{\1}{\mathbbm{1}}

\newcommand{\tr}{\mathrm{tr}}

\renewcommand{\Re}{\mathrm{Re}}
\renewcommand{\Im}{\mathrm{Im}}
\newcommand{\sa}{\operatorname{sa}}

\begin{document}
\title[Noncommutative Majorization]{Noncommutative majorization}

\author[M. Kennedy]{Matthew Kennedy}
\address{Department of Pure Mathematics\\ University of Waterloo\\
Waterloo, Ontario \; N2L 3G1 \\Canada}
\email{matt.kennedy@uwaterloo.ca}

\author[P. Skoufranis]{Paul Skoufranis}
\address{Department of Mathematics and Statistics\\ York University\\
Toronto, Ontario \: M3J 1P3 \\Canada}
\email{pskoufra@yorku.ca}

\begin{abstract}
We introduce a theory of noncommutative majorization that extends the classical majorization theory introduced by Hardy, Littlewood and P\'{o}lya to tuples of self-adjoint matrices that do not necessarily commute. We define and characterize a noncommutative majorization order that extends the classical majorization order. The definition is in terms of convex noncommutative functions, and we utilize the noncommutative convexity theory and noncommutative Choquet theory recently introduced by Davidson and the first author. As an application, we obtain a new necessary and sufficient condition for the existence of a trace-preserving completely positive map, i.e. a quantum channel, that interpolates between two finite sets of matrices. We give examples demonstrating that it is not always possible for this map to be chosen mixed unitary, even locally. We also address the computational difficulty of verifying the noncommutative majorization order. Our results further apply beyond the tracial case, and we obtain more general results characterizing the existence of unital completely positive maps that preserve an arbitrary faithful state or even unital completely positive map.
\end{abstract}

\subjclass[2020]{}
\thanks{First author partially supported by research grant from NSERC (Canada). Second author partially supported by research grant from NSERC (Canada). }
\maketitle
\tableofcontents

\section{Introduction}

The mathematical theory of majorization was introduced and studied by Hardy, Littlewood and P\'{o}lya in their book {\em Inequalities} \cite{HLP1952}, in order to synthesize a number of seemingly disparate ideas appearing in the literature at the time. The theory has subsequently become an important component of convex analysis, with numerous applications throughout mathematics, economics and, more recently, quantum information theory.

In this paper, we will introduce a notion of noncommutative majorization that considerably extends the scope of classical majorization theory. Inspired by the deep relationship between classical majorization theory and classical Choquet theory, we will utilize noncommutative convexity theory and noncommutative Choquet theory, which were recently introduced by Davidson and the first author in \cite{DK2019}. In order to increase the accessibility of this paper, we have included some exposition of this material.

The starting point for classical majorization theory is the majorization order. Originally defined for real vectors, the definition naturally extends to complex self-adjoint matrices. For $n \in \bN$, let $\M_n$ denote the space of $n \times n$ complex matrices, and let $\tr_n : \M_n \to \bC$ denote the normalized trace, i.e. the trace satisfying $\tr_n(1_n) = 1$, where $1_n \in \M_n$ is the identity matrix.

\begin{defn} \label{intro-defn:majorization}
For $n \in \bN$, let $a,b \in \M_n$ be self-adjoint matrices and let $C \subseteq \mathbb{R}$ be a compact convex set containing the eigenvalues of $a$ and $b$. Then $a$ is {\em majorized} by $b$, written $a \prec b$, if
\[
\tr_n(f(a)) \leq \tr_n(f(b))
\]
for every continuous convex function $f : C \to \bR$.
\end{defn}

There are many equivalent characterizations of the majorization order, which is one important reason for the utility of majorization theory.  We now recall two of the most important of these equivalent characterizations. The first, which is in terms of trace-preserving unital completely positive maps, i.e. bistochastic quantum channels, is due to Uhlmann \cite{U1970}. The second, which is in terms of unitary orbits, is an easy consequence of the Birkhoff-von Neumann Theorem and the above-mentioned results in \cite{HLP1952}.

\begin{thm} \label{intro-thm:classical-characterizations}
For $n \in \bN$ and self-adjoint matrices $a,b \in \M_n$, the following are equivalent:
\begin{enumerate} 
\item $a \prec b$;
\item There is a trace-preserving unital completely positive map, i.e. a bistochastic quantum channel, $\Phi : \M_n \to \M_n$ such that $\Phi(b) = a$;
\item The matrix $a$ belongs to the convex hull of the unitary orbit of $b$, i.e. there is $k \in \bN$ and unitary matrices $u_1,\ldots,u_k \in \M_n$ such that
\[
a = \frac{1}{k} \sum_{i=1}^k u_i b u_i^*.
\]
\end{enumerate}
\end{thm}

One of our primary motivations for introducing a noncommutative theory of majorization is to generalize the equivalence between (1) and (2) in Theorem~\ref{intro-thm:classical-characterizations}. In one direction, we seek to replace single self-adjoint matrices with tuples of self-adjoint matrices that do not necessarily commute. In another direction, we seek to replace trace-preserving unital completely positive maps with unital positive maps that preserve an arbitrary given faithful unital completely positive map. Specifically, we are motivated by the following problem. 
For $m,n \in \bN$, recall that a positive linear map $\psi : \M_n \to \M_m$ is {\em faithful} if $\ker \psi \cap \M_n^+ = 0$, where $\M_n^+$ denotes the set of positive matrices in $\M_n$.  

\begin{problem} \label{intro-problem:motivation}
For $m,n \in \bN$, let $\psi : \M_n \to \M_m$ be a faithful unital completely positive map. For $k \in \bN$, let $a = (a_1,\ldots,a_k), b = (b_1,\ldots,b_k) \in \M_n^k$ be $k$-tuples of self-adjoint matrices and let $\B \subseteq \M_n$ denote the unital C*-algebra generated by $b_1,\ldots,b_k$. When does there exist a unital completely positive map $\Phi : \B \to \M_n$ such that $\psi \circ \Phi = \psi$ and $\Phi(b_i) = a_i$ for all $1 \leq i \leq k$? 
\end{problem}

If $\psi$ is the normalized trace $\tr_n$ on $\M_n$, then it follows from precomposing $\Phi$ with a trace-preserving conditional expectation onto $\B$ that Problem \ref{intro-problem:motivation} is equivalent to asking for a trace-preserving unital completely positive map $\Phi : \M_n \to \M_n$ such that $\Phi(b_i) = a_i$ for all $1 \leq i \leq k$. This version of the problem has numerous applications in quantum information theory, and has recently a received significant amount of attention. Note that for $k = 1$, classical majorization theory in the form of Theorem \ref{intro-thm:classical-characterizations}, provides a beautiful solution: $\Phi$ exists if and only if $a \prec b$, and, by (3), $\Phi$ can be taken in a particularly nice form. On the other hand, for $k \geq 2$, there are two solutions to this problem with a much different flavour. It was shown in \cite{HKM2017} that this problem can be formulated as a semidefinite program, and a solution was obtained in \cite{GJBDM2018} in terms of an infinite family of entropy conditions on $a$ and $b$.

We will provide a new solution to the trace-preserving version of this problem for all $k \in \bN$ that is conceptually much closer to the solution provided by classical majorization theory for the case $k = 1$. More generally, utilizing noncommutative majorization theory, we will provide a complete solution to Problem~\ref{intro-problem:motivation}. Specifically, we will introduce and characterize a noncommutative majorization order on tuples of self-adjoint matrices defined in terms of inequalities satisfied by certain convex noncommutative functions. We now outline these ideas in more detail.

The necessary machinery for developing a noncommutative majorization theory is provided by the noncommutative convexity and Choquet theory developed by Davidson and the first author in \cite{DK2019}. In this work, a great deal of classical convexity theory and classical Choquet theory is extended to the noncommutative setting. Importantly, there is an appropriate definition of noncommutative convex function, defined on a compact noncommutative convex set containing tuples of self-adjoint matrices.

Let $a = (a_1,\ldots,a_k), b = (b_1,\ldots,b_k) \in \M_n^k$ be $k$-tuples of self-adjoint matrices as above. There is a natural compact noncommutative convex set $K$ containing $a$ and $b$, namely the closed noncommutative $k$-ball of an appropriately large radius $r > 0$. This is the graded set $K = \sqcup_{m \leq \infty} K_m$, where $K_m \subseteq \M_m^k$ is defined by
\[
K_m = \left\{x = (x_1,\ldots,x_k) \in \M_m^k : \sum_{i=1}^k x_i x_i^* \leq r^2 1_m \right\}.
\]
It is important in noncommutative convexity theory that $m = \infty$ is permitted, in which case $\M_\infty$ denotes the space of bounded operators on a fixed Hilbert space of countably infinite dimension. The set $K$ is a compact noncommutative convex set, meaning that each $K_m$ is compact with respect to a natural topology on $\M_m^k$, and $K$ is closed under noncommutative convex combinations, in the sense that
\[
\sum_{i=1}^\infty \alpha_i^* x_i \alpha_i \in K_n
\]
for every $m \leq \infty$, every bounded sequence $\{x_i \in K_{m_i}\}$, and sequence $\{\alpha_i \in \M_{m_i,m}\}$ satisfying $\sum_i \alpha_i^* \alpha_i = 1_m$.

The definition of a noncommutative function in this setting closely resembles the definition of an noncommutative (free) analytic functions introduced by Taylor and Voiculescu. Let $\M = \sqcup_{m \leq \infty} \M_m$ denote the ``matrix universe,'' which plays the role of the complex numbers in the noncommutative setting. A noncommutative function on $K$ is a function $f : K \to \M$ satisfying certain invariance properties that reflect the noncommutative structure of $\M$. 

For example, let $K$ be the closed noncommutative $k$-ball of radius $r$ as above. Then the affine coordinate functions $h_1,\ldots,h_k : K \to \M$ defined by
\[
h_i(x) = x_i, \quad x = (x_1,\ldots,x_k) \in K
\]
for $1 \leq i \leq k$ are noncommutative functions on $K$. Let $\1 : K \to \M$ denote the constant noncommutative function on $K$ defined by $\1(x) = 1_n$ for $m \leq \infty$ and $x \in K_n$. The set of noncommutative polynomials in $\1,h_1,\ldots,h_k$ is dense in the set of all bounded noncommutative functions on $K$, in a certain precise sense.

There are natural definitions of continuity and lower semicontinuity for noncommutative functions, as well as a natural definition of convexity. Specifically, a self-adjoint noncommutative function is convex if its epigraph is convex, which is equivalent to satisfying a noncommutative analogue of Jensen's inequality. For example, the affine coordinate functions $h_1,\ldots,h_k$ and their squares are continuous and convex. However, higher powers of these functions are not convex.

The noncommutative setting additionally requires us to consider higher order noncommutative functions on $K$. For $p \leq \infty$, a noncommutative function of order $p$ on $K$ is a function $f : K \to \M_p(\M)$, which can be identified with a $p \times p$ matrix of noncommutative functions on $K$.

We will be interested in the noncommutative functions that arise as noncommutative suprema of higher order continuous affine nc functions. These functions, which we will refer to as noncommutative majorization functions, are noncommutative analogues of lower semicontinuous convex functions in classical convexity theory. In the noncommutative setting, they are also well-behaved, although they may be multivalued because of the anti-lattice structure of the self-adjoint operators. The noncommutative majorization order is defined in terms of these functions.

\begin{defn} \label{intro-defn:nc-majorization}
For $m,n \in \bN$, let $\psi : \M_n \to \M_m$ be a faithful unital completely positive map. For $k \in \bN$ and $k$-tuples $a = (a_1,\ldots,a_k), b = (b_1,\ldots,b_k) \in \M_n^k$, we will say that $a$ is $\psi$-majorized by $b$ and write $a \prec_\psi b$ if
\[
(\id_{\M_p} \otimes \psi)(F(a)) \leq (\id_{\M_p} \otimes \psi)(F(b))
\]
for every $p \leq \infty$ and every noncommutative majorization function $F : K \to \M_p(\M)$, where $K$ is any closed nc $k$-ball containing $a$ and $b$. If $\psi$ is the normalized trace $\tr_n$ on $\M_n$, then we will simply say that $a$ is {\em majorized} by $b$ and write $a \prec b$.
\end{defn}

Higher order multivalued noncommutative functions play an important role in noncommutative convexity theory and noncommutative Choquet theory. However, it is natural to wonder if they are truly necessary in the definition of the noncommutative majorization order. We will demonstrate in Section \ref{sec:necessity} that this question has an affirmative answer in a very strong sense.

The following result provides a complete solution to Problem \ref{intro-problem:motivation} in terms of the appropriate noncommutative majorization order.

\begin{thm} \label{intro-thm:main}
For $m,n \in \bN$, let $\psi : \M_n \to \M_m$ be a faithful unital completely positive map. For $k \in \bN$, let $a = (a_1,\ldots,a_k), b = (b_1,\ldots,b_k) \in \M_n^k$ be $k$-tuples of self-adjoint matrices, and let $\mathcal{B} \subseteq \M_n$ denote the unital C*-algebra generated by $b_1,\ldots,b_k$. The following are equivalent:
\begin{enumerate}
\item $a \prec_\psi b$;
\item There is a unital completely positive map $\Phi : \mathcal{B} \to \M_n$ such that $\psi \circ \Phi = \psi$ and $\Phi(b_i) = a_i$ for all $1 \leq i \leq k$.
\end{enumerate}
\end{thm}

As an immediate Corollary, we obtain a natural generalization of the equivalence between (1) and (2) in Theorem \ref{intro-thm:classical-characterizations}. The following result completely characterizes interpolation of tuples of self-adjoint matrices by trace-preserving unital completely positive maps, i.e. bistochastic quantum channels.

\begin{cor} \label{intro-cor:nc-majorization}
For $k,n \in \bN$ and $k$-tuples of self-adjoint matrices $a = (a_1,\ldots,a_k),b=(b_1,\ldots,b_k) \in \M_n^k$, the following are equivalent:
\begin{enumerate}
\item $a \prec b$;
\item There is a trace-preserving unital completely positive map, i.e. a bistochastic quantum channel, $\Phi : \M_n \to \M_n$ such that $\Phi(b_i) = a_i$ for all $1 \leq i \leq k$.  
\end{enumerate}
\end{cor}

We also obtain a non-unital extension of Corollary~\ref{intro-cor:nc-majorization} that completely characterizes interpolation of tuples of self-adjoint matrices by trace-preserving completely positive maps, i.e. (not necessarily unital) quantum channels. In addition, the matrices are not required to have the same dimension. The key ideas underlying this result are generalizations of Definition~\ref{intro-defn:nc-majorization} and Theorem~\ref{intro-thm:main}.

\begin{thm} \label{intro-thm:non-unital}
For $k,n_1,n_2 \in \bN$ and $k$-tuples of self-adjoint matrices $a = (a_1,\ldots,a_k) \in \M_{n_1}^k$ and $b = (b_1,\ldots,b_k) \in \M_{n_2}^k$, the following are equivalent:
\begin{enumerate}
\item $(\id_{\M_p} \otimes \tr_{n_1})(F(a')) \leq (\id_{\M_p} \otimes \omega)(F(b'))$ for every $p \leq \infty$ and every nc majorization function $F : K \to \M_p(\M)$, where $a' \in \M_{n_1}^k$ and $b' \in \M_{n_2 + 1}$ are defined by
\[
a' = \frac{1}{n_1} a, \quad b' = (b_1 \oplus 0, \ldots, b_k \oplus 0),
\]
and $\omega : \M_{n_2 + 1} \to \bC$ is the state defined by
\[
\omega([x_{ij}]) = \frac{1}{n_1} \sum_{i=1}^{n_2} x_{ii} + \frac{n_1 - 1}{n_1} x_{n_2 + 1, n_2 + 1};
\]
\item There is a trace-preserving completely positive map, i.e. a quantum channel, $\Phi : \M_{n_2} \to \M_{n_1}$ such that $\Phi(b_i) = a_i$ for all $1 \leq i \leq k$. 
\end{enumerate}
\end{thm}

We expect that it will be important to understand the computational difficulty of verifying Definition \ref{intro-defn:nc-majorization}. We will address this point in Section~\ref{sec:effective} by proving that it suffices to consider noncommutative functions obtained as the supremum of finitely many continuous affine functions of order at most $2n_2^2 + 2$.  

We will also address the question of whether there is a noncommutative generalization of the equivalence between (1) and (3) in Theorem~\ref{intro-thm:classical-characterizations}. For $n \in \bN$ and self-adjoint matrices $a,b \in \M_n$ with $a \prec b$, (3) in Theorem~\ref{intro-thm:classical-characterizations} implies that there is a mixed unitary quantum channel $\Phi : \M_n \to \M_n$, a particularly nice trace-preserving unital completely positive maps arising as convex combinations of unitary conjugations, such that $\Phi(b) = a$. Consequently, in Section \ref{sec:mixed}, we will consider whether the unital completely positive maps obtained from Corollary \ref{intro-cor:nc-majorization} can always be chosen to be mixed unitary. However, we will prove that for all $n \geq 4$ and $k \geq 2$, there are tuples of self-adjoint matrices $a = (a_1,\ldots,a_k), b = (b_1,\ldots,b_k) \in \M_n^k$ such that $a \prec b$, but there is no mixed unitary quantum channel $\Phi : \M_n \to \M_n$ satisfying $\Phi(b_i) = a_i$ for all $1 \leq i \leq k$. This is in spite of Corollary \ref{intro-cor:nc-majorization}, thereby demonstrating that majorization is a theory of quantum channels that need not be mixed unitary, i.e. the $k = 1$ case is very special.

In addition to this introduction, this paper has 8 more sections. Section~\ref{sec:nc-convexity}, Section~\ref{sec:nc-functions} and Section~\ref{sec:multivalued-nc-functions} provide the required background on noncommutative convexity and noncommutative function theory. In Section \ref{sec:nc-majorization}, we introduce the noncommutative majorization orders and prove Theorem~\ref{intro-thm:main} and Corollary~\ref{intro-cor:nc-majorization}. In Section \ref{sec:quantum-channel-interpolation}, we prove Theorem~\ref{intro-thm:non-unital}. In Section \ref{sec:effective}, we address the computational difficulty of verifying the noncommutative majorization orders, and in Section \ref{sec:necessity}, we address the necessity of considering higher order multivalued functions. Finally, in Section~\ref{sec:mixed}, we demonstrate that the maps arising from Corollary~\ref{intro-cor:nc-majorization} can not be mixed unitary in general.

In this paper, we have restricted our attention to the finite-dimensional setting in order to make the results accessible to a broad audience. We will consider the infinite-dimensional setting in a forthcoming paper.

\section{Noncommutative convexity} \label{sec:nc-convexity}

In this section, we will briefly review the theory of noncommutative convexity. We refer the interested reader to the monograph \cite{DK2019} and the survey \cite{DK2024} for more details. For the sake of simplicity, we will specialize the theory as much as possible here, taking advantage of the fact that we are primarily interested in finite tuples of matrices. For brevity, we will abbreviate noncommutative with nc when appropriate.

The notion of an nc convex set is the starting point for the theory of nc convexity. Before introducing the definition, we will require some notational conventions.

For $m,n \leq \infty$, let $\M_{m,n}$ denote the vector space of bounded linear maps between fixed complex Hilbert spaces $H_n$ and $H_m$. We equip $\M_{m,n}$ with the weak* topology, with agrees with the norm topology when $m,n < \infty$. By fixing orthonormal bases for $H_n$ and $H_m$, we can identify $\M_{m,n}$ with $m \times n$ matrices over $\bC$. We will write $\M_n$ for $\M_{n,n}$, $(\M_n)_{\sa}$ for the self-adjoint matrices $\M_n$, and $\M_n^+$ for the positive matrices in $\M_n$. Let $1_n$ denote the identity element of $\M_n$. For $n < \infty$, let $\tr_n : \M_n \to \bC$ denote the trace on $\M_n$, normalized so that $\tr_n(1_n) = 1$.

For $k \in \bN$ and $n \leq \infty$, we will write $\M_n(\bC^k)$ for the tensor product $\M_n \otimes \bC^k \cong \M_n^k$, equipped with the product topology. It will be convenient to view elements of $\M_n(\bC^k)$ as $k$-tuples of elements in $\M_n$.  The elements of $\M_{m,n}$ and $\M_{n,m}$ naturally act on $\M_n^k$ by left and right multiplication respectively. Specifically, for $x = (x_1,\ldots,x_k) \in \M_n^k$, $\alpha \in \M_{m,n}$ and $\beta \in \M_{n,m}$,
\[
\alpha x \beta = (\alpha x_1 \beta,\ldots,\alpha x_k \beta) \in \M_m^k.
\]

\begin{defn}
For $k \in \bN$, an {\em nc convex set} over $\bC^k$ is a graded set $K = \sqcup_{n \leq \infty} K_n$, with $K_n \subseteq \M_n^k$ for each $n$, that is closed under noncommutative convex combinations, meaning that $\sum_{i=1}^\infty \alpha_i^* x_i \alpha_i \in K_n$ for every bounded sequence $\{x_i \in K_{n_i}\}$ and every sequence $\{\alpha_i \in \M_{n_i,n}\}$ satisfying $\sum_{i=1}^\infty \alpha_i^* \alpha_i = 1_n$. We will say that $K$ is {\em closed} if each $K_n \subseteq \M_n^k$ is closed, and {\em compact} if each $K_n \subseteq \M_n^k$ is compact.
\end{defn}

\begin{rem}
By \cite{DK2019}*{Proposition 2.2.8}, the condition that $K$ is closed under nc convex combinations is equivalent to the condition that it is closed under direct sums and compressions, meaning that
\begin{enumerate}
\item $\sum_i \alpha_i x_i \alpha_i^* \in K_n$ for every bounded sequence $\{x_i \in K_{n_i}\}$ and every sequence of isometries $\{\alpha_i \in \M_{n, n_i}\}$ satisfying $\sum_{i=1}^\infty \alpha_i \alpha_i^* = 1_n$, and
\item $\beta^* x \beta \in K_m$ for every $x \in K_n$ and every isometry $\beta \in \M_{n,m}$.
\end{enumerate}
The condition that $K$ is compact is equivalent to the condition that each $K_n$ is closed and bounded in $\M_n^k$. It is both intentional and important that $n = \infty$ is allowed in the definitions of an nc convex set and nc convex combinations.
\end{rem}

\begin{rem}
Although we will primarily be interested in nc convex sets over $\bC^k$, we can more generally consider nc convex sets over an operator space $E$.
\end{rem}

\begin{exam}\label{exam:universe}
Let $\M$ denote the ``matrix universe'' $\M = \sqcup_{n \leq \infty} \M_n$. Then $\M$ is a closed but non-compact nc convex set that plays the role of the complex numbers in nc convexity and nc function theory. We will refer to $\M$ throughout this paper.
\end{exam}

For the purpose of this paper, we will primarily be interested in the following particularly nice examples of compact nc convex sets.

\begin{exam} \label{exam:nc-k-ball}
Let $k \in \bN$ and $r > 0$. The {\em closed nc $k$-ball of radius $r$} is the compact nc convex set $K = \sqcup_{n \leq \infty} K_n$ over $\bC^k$ defined for $n \leq \infty$ by
\[
K_n = \left\{x = (x_1,\ldots,x_n) \in \M_n^k : \sum_{i=1}^k x_i^* x_i \leq r^2 1_n \right\}. 
\]
Equivalently, $x = (x_1,\ldots,x_k) \in \M_n^k$ belongs to $K_n$ if and only if $\|x\|_{\operatorname{op}} \leq r$, where $x$ is viewed as an operator from $H_n^k$ to $H_n$ and $\|\cdot\|_{\operatorname{op}}$ denotes the corresponding operator norm.

This is a natural setting for noncommutative majorization theory since, for $n \in \bN$ and $k$-tuples of self-adjoint matrices $a = (a_1,\ldots,a_k), b = (b_1,\ldots,b_k) \in \M_n^k$, both $a$ and $b$ will belong to $K$ if $\max(\|a\|_{\operatorname{op}},\|b\|_{\operatorname{op}}) \leq r$.
\end{exam}

\begin{rem}
For $k \in \bN$, let $E = \sqcup_{n \leq \infty} E_n$ with $E_n\subseteq \M_n^k$. Since the intersection of nc convex sets over $\bC^k$ is an nc convex set over $\bC^k$, there is a smallest nc convex set, say $K$, containing $E$, which we will refer to as the {\em nc convex hull} of $E$. The closure of $K$ is easily seen to be the smallest closed nc convex set containing $E$, and so we will refer to it as the {\em closed nc convex hull} of $E$. 
\end{rem}

\section{Noncommutative functions} \label{sec:nc-functions}

\subsection{Noncommutative functions}

The theory of nc convexity becomes much richer with the addition of an appropriate notion of function. The following definition is \cite{DK2019}*{Definition 4.2.1}.

\begin{defn}
Let $K = \sqcup_{n \leq \infty} K_n$ be a compact nc convex set over $\bC^k$.  A function $f : K \to \M$ is said to be an {\em nc function} if it is graded, respects direct sums, and is equivariant with respect to unitaries, meaning that
\begin{enumerate}
\item $f(K_n) \subseteq \M_n$ for all $n \leq \infty$,
\item $f(\sum_{i=1}^\infty \alpha_i x_i \alpha_i^*) = \sum_{i=1}^\infty \alpha_i f(x_i) \alpha_i^*$ for every bounded sequence $\{x_i \in K_{n_i}\}$ and every sequence of isometries $\{\alpha_i \in \M_{n, n_i}\}$ satisfying $\sum_{i=1}^\infty \alpha_i \alpha_i^* = 1_n$, and
\item $f(\beta^* x \beta) = \beta^* f(x) \beta$ for every $x \in K_n$ and every unitary $\beta \in \M_n$.
\end{enumerate}
The function $f$ is said to be an {\em affine} nc function if, in addition, it is equivariant with respect to isometries, meaning that
\begin{enumerate}
\item[(4)] $f(\gamma^* x \gamma) = \gamma^* f(x) \gamma$ for every $x \in K_n$ and every isometry $\gamma \in \M_{n,m}$.
\end{enumerate}
\end{defn}

For $k \in \bN$, let $K$ be a compact nc convex set over $\bC^k$. There is a natural {\em uniform norm} $\|\cdot\|_\infty$ on the nc functions on $K$. For an nc function $f : K \to \M$, 
\[
\|f\|_\infty := \sup_{x \in K} \|f(x)\|.
\]
We will say that $f$ is {\em bounded} if $\|f\|_\infty < \infty$. In \cite{DK2019}*{Proposition 4.2.4}, it was shown that the space $\rB(K)$ of {\em bounded nc functions} on $K$ is a von Neumann algebra with respect to the uniform norm and pointwise multiplication, with the adjoint given by $f^*(x) = f(x)^*$ for $x \in K$.

An affine nc function $h : K \to \M$ is {\em continuous} if the restriction $h|_{K_n} : K_n \to \M_n$ is continuous for all $n \leq \infty$. The operator subsystem $\rA(K) \subseteq \rB(K)$ consisting of continuous affine nc functions plays a key role in the theory.

\begin{exam} \label{exam:nc-k-ball-functions}
For $k \in \bN$ and $r > 0$, let $K = \sqcup_{n \leq \infty} K_n$ be the closed nc $k$-ball of radius $r$ from Example~\ref{exam:nc-k-ball}. For $1 \leq i \leq k$, define the nc coordinate function $h_i : K \to \M$ by
\[
h_i(x) = x_i, \quad x = (x_1,\ldots,x_k) \in K.
\]
Each $h_i$ is an affine nc function and $\rA(K) = \operatorname{span}\{\1, h_1,\ldots,h_k\}$, where the scalar nc function $\1 : K \to \M$ is defined by by $\1(x) = 1_n$ for $x \in K_n$. Indeed, this follows from the fact that $K_1$ is the closed ball of radius $r$ in $\bC^k$, and the restriction map from $\rA(K)$ to the (ordinary) continuous affine functions on $K_1$ is a unital order isomorphism.
\end{exam}

The C*-subalgebra $\rC(K) \subseteq \rB(K)$ generated by $\rA(K)$ is the C*-algebra of {\em continuous nc functions} on $K$, and functions in $\rC(K)$ are called {\em continuous nc functions}. This name is appropriate by \cite{DK2019}*{Theorem 4.4.2}, which can be viewed as a kind of noncommutative Stone-Weierstrass-type theorem. It asserts that a bounded nc function in $\rB(K)$ belongs to $\rC(K)$ precisely when it is continuous with respect to the weakest topology on $K$ making each element of $\rA(K)$ point-ultrastrong* continuous.

The interplay between the structure of $\rA(K)$ and the representation theory of $\rC(K)$ is the starting point for nc Choquet theory. It follows from the categorical duality between operator systems and compact nc convex sets that we can identify $K$ with the nc state space of $\rA(K)$, i.e. the unital completely positive maps from $\rA(K)$ to $\M_n$ for $n \leq \infty$ (see \cite{DK2019}*{Theorem 3.2.5}). Specifically, if $\mu : \rA(K) \to \M_n$ is a unital completely positive map, then there is a point $x \in K_n$ such that $\mu(a) = a(x)$ for all $a \in \rA(K)$.

For $x \in K_n$, let $\delta_x : \rB(K) \to \M_n$ denote the evaluation map at $x$, so that $\delta_x(f) = f(x)$ for $f \in \rB(K)$. Then $\delta_x$ is a normal unital *-homomorphism, and every normal unital *-homomorphism from $\rC(K)$ to $\M_n$ for $n \leq \infty$ is of this form by \cite{DK2019}*{Proposition 4.3.2}. In fact, it follows from the universal properties of $\rC(K)$ and $\rB(K)$ that the restriction $\delta_x|_{\rC(X)}$ is the unique extension of $x$ to a unital *-homomorphism on $\rC(K)$, and $\delta_x$ is the unique extension of $\delta_x|_{\rC(X)}$ to a normal unital *-homomorphism on $\rB(K)$ (see \cite{DK2019}*{Theorem 4.5.1} and \cite{DK2019}*{Theorem 4.3.3}).

Although every nc state $x \in K$ has a unique extension to a unital *-homomorphism on $\rC(K)$, there may be many unital completely positive extensions that are not *-homomorphisms. Note that if $\mu : \rC(K) \to \M_n$ is a unital completely positive map, then the restriction $\mu|_{\rA(K)}$ is a point in $K_n$.

\begin{defn}
    Let $K$ be a compact nc convex set and let $\mu : \rC(K) \to \M_n$ be a unital completely positive map. The {\em barycentre} of $\mu$ is the restriction $\mu|_{\rA(K)} \in K_n$, and the map $\mu$ is said to {\em represent}  $x$.
    Equivalently, $\mu$ represents $x$ if $\mu(a) = a(x)$ for all $a \in \rA(K)$.  
\end{defn}

The following definition from \cite{DK2019}*{Definition 5.2.1} encodes the notion of Stinespring representation of a unital completely positive map. 

\begin{defn}
Let $K$ be a compact nc convex set and let $\mu : C(K) \to \M_m$ be a unital completely positive map. A pair $(x, \alpha) \in K_n \times \M_{n,m}$ with $\alpha$ an isometry is a \emph{representation} of $\mu$ if $\mu = \alpha^* \delta_x \alpha$. If, in addition, $\{f(x) \alpha H_m : f \in \rC(K)\}$ is dense in $H_n$, then we will say that $(x,\alpha)$ is {\em minimal}.
\end{defn}

Stinespring's theorem, combined with the representation theory of $\rC(K)$, implies the existence of a minimal representation for every unital completely positive map from $\rC(K)$ to $\M_n$ for $n \leq \infty$. Furthermore, this minimal representation is unique up to unitary equivalence.

\subsection{Higher order noncommutative functions}

The nc majorization order will require us to consider higher order nc functions, which can be viewed as matrices of nc functions. For $k \in \bN$, let $K$ be a compact nc convex set over $\bC^k$. For $p \leq \infty$, identify $\M_p \otimes \rB(K)$ with $\M_p(\rB(K))$. We can view $f \in \M_p(\rB(K))$ as a function $f : K \to \M_p(\M)$, and it easy to verify that $f$ is graded and respects direct sums. Furthermore, $f$ is unitarily equivariant, in the sense that for $n \leq \infty$, $x \in K_n$, and a unitary $\alpha \in \M_n$,
\[
f(\alpha^* x \alpha) = (1_p \otimes \alpha^*) f(x) (1_p \otimes \alpha). 
\]
We will refer to elements of $\M_p(\rA(K))$, $\M_p(\rC(K))$ and $\M_p(\rB(K))$ as {\em continuous affine}, {\em continuous}, and {\em bounded} nc functions respectively. We will say that $f \in \M_p(\rB(K))$ is {\em self-adjoint} if $f = f^*$, which is equivalent to $f(x)$ being self-adjoint for every $x \in K$.

\begin{exam}
For $k \in \bN$ and $r > 0$, let $K = \sqcup_{n \leq \infty} K_n$ be the closed nc $k$-ball of radius $r$ from Example~\ref{exam:nc-k-ball}. Let $h_1,\ldots,h_k \in \rA(K)$ be the coordinate affine nc functions and $\1$ the scalar affine nc function from Example~\ref{exam:nc-k-ball-functions}. We saw that $\rA(K) = \operatorname{span}\{\1,h_1,\ldots,h_k\}$. It follows from this that for $p \leq \infty$, every continuous affine nc function $a \in \M_p(\rA(K))$ can be written as 
\[
a = \alpha_0 \otimes \1 + \alpha_1 \otimes h_1 + \cdots + \alpha_k \otimes h_k,
\]
for some $\alpha_0,\ldots,\alpha_k \in \M_p$. 
\end{exam}

\subsection{Convex noncommutative functions}

The following definition of a convex nc function is analogous to the classical definition of a convex function. Note that for $k \in \bN$, if $K$ is a compact nc convex set over $\bC^k$, then for $p \leq \infty$, $\sqcup_{n \leq \infty} K_n \times \M_p(\M_n)$ is a nc convex set over the operator space $\bC^k \times \M_p$. The following definition is \cite{DK2019}*{Definition 7.2.1}.

\begin{defn}[\cite{DK2019}*{Definition 7.2.1}]
Let $K$ be a compact nc convex set. For $p \leq \infty$ and a bounded nc function $f \in \M_p(\rB(K))$, the \emph{epigraph of $f$} is the subset
\[
\operatorname{Epi}(f) \subseteq \sqcup_{n \leq \infty} K_n \times \M_p(\M_n)
\]
defined by
\[
\operatorname{Epi}(f) = \{(x,\alpha) \in K \times \M_p(\M) : x \in K,\ \alpha \geq f(x)\}.
\]
The nc function $f$ is {\em convex} if $\operatorname{Epi}(f)$ is a nc convex set, and {\em lower semicontinuous} if $\operatorname{Epi}(f)$ is closed.
\end{defn}

\begin{rem}
By \cite{DK2019}*{Remark 7.2.2}, for $p \leq \infty$, a self-adjoint bounded nc function $f \in \M_p(B(K))$ is convex if and only if $f$ satisfies the noncommutative Jensen-type inequality
\[
f(\alpha^*x\alpha) \leq (1_p \otimes \alpha^*) f(x) (1_p \otimes \alpha)
\]
for every $m,n \leq \infty$, every $x \in K_n$ and every isometry $\alpha \in \M_{n,m}$. Furthermore, by \cite{DK2019}*{Proposition 7.2.3}, $f$ is convex if and only if
\[
f(t x + (1-t) y) \leq t f(x) + (1-t)f(y)
\]
for every $n \leq \infty$, every $x, y \in K_n$, and every $t \in [0,1]$.
\end{rem}

\section{Multivalued noncommutative functions} \label{sec:multivalued-nc-functions}

In nc convexity theory and nc Choquet theory, it is frequently necessary to work with certain well-behaved multivalued nc functions, as we will now explain.

The notion of a convex envelope plays a critical role in classical Choquet theory. For a bounded real-valued function on a compact convex set, the convex envelope is the best approximation from below by a lower semicontinuous convex function. By definition, the convex envelope is the supremum of the continuous affine functions dominated by the function. Equivalently, the epigraph of the convex envelope is the closed convex hull of the epigraph of the function.

There is a notion of convex envelope that plays an analogous role in nc Choquet theory. For a self-adjoint bounded nc function on a compact nc convex set, the convex envelope is similarly the best approximation from below by a lower semicontinuous convex nc function in an appropriate sense. However, it is not possible to define the convex envelope as a supremum of continuous affine nc functions, since such a supremum will generally not exist. This can be seen from Kadison's anti-lattice theorem in \cite{K1951}, which asserts that an incomparable pair of bounded operators do not have a supremum.

In order to circumvent this difficulty, the convex envelope of an nc function is defined in terms of its epigraph. Specifically, the epigraph of the convex envelope is defined to be the closed nc convex hull of the epigraph of the function. Equivalently, the epigraph of the convex envelope is the intersection of the epigraphs of the continuous affine nc functions dominated by the function.

While the convex envelope of an nc function may not be a single-valued nc function, it will always be a well-behaved multivalued nc function, as we will now make precise. The following definition is from \cite{DK2019}*{Definition 7.3.1}.

\begin{defn} \label{defn:multivalued-nc-function}
Let $K$ be an nc convex set. For $p \leq \infty$, a self-adjoint multivalued function $F : K \to \M_p(\M)$ is a \emph{multivalued nc function} if it is non-degenerate, graded, unitarily equivariant and upwards directed, meaning that
\begin{enumerate}
\item $F(x) \neq \emptyset$ for all $x \in K$,
\item $F(K_n) \subseteq \M_p(\M_n)$ for every $n \leq \infty$,
\item $F(\alpha^* x \alpha) = (1_p \otimes \alpha)^* F(x) (1_p \otimes \alpha)$ for every $n \leq \infty$, every $x \in K_n$, and every unitary $\alpha \in \M_n$, and
\item $F(x) = F(x) + \M_p(\M_n)^+$ for every $n \leq \infty$ and every $x \in K_n$; that is, $F(x)$ is directed upwards.
\end{enumerate}
If $G : K \to \M_p(\M)$ is another multivalued nc function, then we write $F \leq G$ if $F(x) \supseteq G(x)$ for all $x \in K$.

A multivalued nc function $F$ is \emph{bounded} if there exists $\lambda > 0$ such that for every $x \in K$ and every $\beta \in F(x)$ there exists an $\alpha \in F(x)$ with $\alpha \leq \beta$ and $\|\alpha\| \leq \lambda$.  If $F$ is bounded, then $\|F\|$ denotes the infimum over all such $\lambda$. 
\end{defn}

\begin{rem}
If $F : K \to \M_p(\M)$ is a multivalued nc function, then since $F$ is upwards directed, it is completely determined by the set of minimal elements in $F(x)$ for $x \in K$.
\end{rem}

There are natural definitions of convexity and lower semicontinuity for multivalued nc functions that extend the definitions of convexity and lower semicontinuity for nc functions. For these, we require the definition of the graph of a multivalued nc function.

\begin{defn}[\cite{DK2019}*{Definition 7.3.2}]
Let $K$ be a compact nc convex set. For $p \leq \infty$, let $F : K \to \M_p(\M)$ be a multivalued nc function.  The \emph{graph} of $F$ is the subset
\[
\mathrm{Graph}(F) \subseteq \bigsqcup_{n \leq \infty} K_n \times \M_p(\M_n)
\]
defined by
\[
\mathrm{Graph}_n(F) = \{(x, T) \in K_n \times \M_p(\M_n) \, \mid \, x \in K_n \text{ and } T \in F(x)\}.
\]
A multivalued nc function $F$ is \emph{convex} if $\mathrm{Graph}(F)$ is an nc convex set, and \emph{lower semicontinuous} if $\mathrm{Graph}(F)$ is closed.
\end{defn}

\begin{rem}\label{rem:nc-function-to-multivalued}
Let $K$ be a compact nc convex set. For $p \leq \infty$, a (single-valued) self-adjoint bounded nc function $f \in \M_p(\rB(K))$ can be identified with a bounded multivalued nc function $F : K \to \M_p(\M)$ defined by $F(x) = [f(x),\infty)$ for $n \leq \infty$ and $x \in K_n$, where
\[
[f(x),\infty) = \{\alpha \in \M_p(\M_n)_{\mathrm{sa}} : \alpha \geq f(x) \}.
\]
Equivalently, $F$ is the unique bounded multivalued nc function satisfying $\operatorname{Graph}(F) = \operatorname{Epi}(f)$. It follows immediately from this that $F$ is convex if and only if $f$ is convex.

Note that for $x \in K$, $f(x)$ is the unique minimal element in $F(x)$. Conversely, if  $F : K \to \M_m(\M)$ is a bounded multivalued nc function with the property that $F(x)$ has a unique minimal element for each $x \in K$, then $F$ arises from a self-adjoint bounded nc function in this way.
\end{rem}

Let $K$ be a compact nc convex set. For $p \leq \infty$, let $F : K \to \M_p(\M)$ be a multivalued nc function and let $g : K \to \M_p(\M)$ be an nc function. We will frequently identify $g$ with the corresponding multivalued function $G$ as in Remark \ref{rem:nc-function-to-multivalued} and write $F \leq g$ if $F \leq G$.

\begin{exam}\label{exam:multiple-affine-to-a-single}
Let $K$ be a compact nc convex set. For $p \leq \infty$, let $\{f_i \in \M_p(\rB(K)) \}_{i \in I}$ be a family of bounded nc functions. The {\em supremum} of the $f_i$ is the multi-valued nc function $F : K \to \M_p(\M)$ determined by
\[
\operatorname{Graph}(F) = \bigcap_{i \in I} \operatorname{Epi}(f_i).
\]
In other words, for $n \leq \infty$ and $x \in K_n$,
\[
F(x) = \bigcap_{i \in I} \{\alpha \in \M_p(\M_n)_{\sa} : \alpha \geq f_i(x) \}.
\]

Now suppose that $I$ is finite, say $I = \{1,\ldots,k\}$. Define a bounded nc function $f \in \M_k(\M_p(\rB(K)))$ by $f = \oplus_{i=1}^k f_i$. Then for $n \leq \infty$ and $x \in K_n$, the supremum $F$ defined above satisfies
\[
F(x) = \{\alpha \in \M_p(\M_n)_{\sa} : 1_k \otimes \alpha \geq f(x)\}.
\]

We caution that the supremum $F$ may not be bounded even though the family $\{f_i \in \M_p(\rB(K))\}_{i \in I}$ was assumed to be bounded.
\end{exam}

Recall that if $\mu : C(K) \to \M_m$ is a unital completely positive map with minimal representation $(x,\nu) \in K_n \times \M_{n,m}$, then $\mu(f) = \nu^* f(x) \nu$ for $f \in C(K)$. The map $\mu$ naturally extends to multivalued nc functions as in the following definition.

\begin{defn}[\cite{DK2019}*{Definition 7.3.4}]
    Let $K$ be a compact nc convex set and let $\mu : C(K) \to \M_m$ be a unital completely positive map with minimal representation $(x, \alpha) \in K_n \times \M_{n,m}$. For $p \leq \infty$ and a multivalued nc function $F : K \to \M_p(\M)$, we evaluate $\mu$ at $F$ by
    \[
    \mu(F) = (1_p \otimes \alpha)^* F(x)(1_p \otimes \alpha),
    \]
    where
    \[
    (1_p \otimes \alpha^*) F(x)(1_p \otimes \alpha) = \{(1_p \otimes \alpha^*) \beta (1_p \otimes \alpha) : \beta \in F(x) \}.
    \]
\end{defn}

\begin{rem}
Note that the above definition of $\mu(F)$ does not depend on the choice of representation of $\mu$ because of the unitary equivariance of $F$.
\end{rem}

We are now ready to define the convex envelope of a bounded multivalued nc function.

\begin{defn}[\cite{DK2019}*{Definition 7.4.1}]
Let $K$ be a compact nc convex set. For $p \leq \infty$, let $F : K \to \M_p(\M)$ be a bounded multivalued nc function. The {\em convex envelope} $\overline{F} : K \to \M_p(\M)$ of $F$ is the multivalued nc function with graph equal to the closed convex hull of the graph of $F$.
\end{defn}

\begin{rem}
It follows from \cite{DK2019}*{Proposition 7.4.2} that $\overline{F}$ is a bounded lower semicontinuous convex multivalued nc function with $\overline{F} \leq F$. Furthermore, $F$ is lower semicontinuous and convex if and only if $F = \overline{F}$.
\end{rem}

The following theorem from \cite{DK2019}*{Theorem 7.4.3} is a noncommutative analogue of the fact that the convex envelope of a bounded real-valued function on a compact convex set is the supremum of the continuous affine functions dominated by the function.

\begin{thm} \label{thm:DK-7.4.3}
Let $K$ be a compact nc convex set. For $p \leq \infty$, let $F: K \to \M_p(\M)$ be a bounded multivalued nc function. Then for $n \leq \infty$ and $x \in K_n$,
\[
\overline{F}(x) = \bigcap_q \bigcap_{a\leq 1_q \otimes F} \left\{ \alpha \in \M_p(\M_n)_{sa} : 1_q \otimes \alpha \geq a(x)\right\}
\]
where the intersection is taken over all $q \leq \infty$ and all self-adjoint continuous affine nc functions $a \in \M_q(\M_p(A(K)))$ satisfying $a \leq 1_q \otimes F$.
\end{thm}

The following result connects convex envelopes with unital completely positive maps on $\rC(K)$.

\begin{thm}[\cite{DK2019}*{Theorem 7.5.1}]\label{thm:DK-7.5.1}
Let $K$ be a compact nc convex set. For $p \leq \infty$, let $f : K \to \M_p(\M)$ be a self-adjoint lower semicontinuous bounded nc function with convex envelope $\overline{f}$.  Then for $n \leq \infty$ and $x \in K_n$, 
\[
\overline{f}(x) = \bigcup_{\mu} [\mu(f), \infty)
\]
where the union is taken over all unital completely positive maps $\mu : C(K) \to \M_n$ with barycentre $x$.
\end{thm}

For the most part, multivalued nc functions can be treated like (single-valued) nc functions. However, extra caution is sometime required, as we will now explain.

Let $K$ be a compact nc convex set. For $m,n \leq \infty$, let $\psi : \M_n \to \M_m$ be a unital completely positive map and let $x \in K_n$. Define $\mu : \rC(K) \to \M_m$ by $\mu = \psi \circ \delta_x$. Then $\mu$ is a unital completely positive map. Let $(y,\alpha) \in K_p \times \M_{p,m}$ be a representation of $\mu$. Then for $q \leq \infty$ and $f \in \M_q(\rC(K))$,
\[
\mu(f) = (1_q \otimes \alpha^*) f(x) (1_q \otimes \alpha).
\]

Now let $F : K \to \M_q(\M)$ be a multivalued function. In the next section, we will need to consider expressions of the form
\[
(\id_{\M_q} \otimes \psi)(F(x)) = \{ (\id_{\M_q} \otimes \psi)(\beta) : \beta \in F(x) \}.
\]
Despite the fact that $\mu = \psi \circ \delta_x$, this expression does not necessarily agree with the expression
\[
(\id_{\M_q} \otimes \mu)(F) = (1_q \otimes \alpha^*) F(x) (1_q \otimes \alpha),
\]
as the following example will show.

\begin{exam}\label{exam:difference-between-multivalued-function-evaluation}
Define $a_1,a_2 \in \M_2$ by
\[
a_1 = \left[\begin{matrix}1 & 0 \\ 0 & 0\end{matrix}\right], \quad a_2 = \left[\begin{matrix}1/2 & 1/2 \\ 1/2 & 1/2\end{matrix}\right].
\]
Let $K$ be a closed nc $2$-ball containing the tuple $a = (a_1,a_2) \in \M_2^2$, and let $h_1,h_2 \in \rA(K)$ denote the coordinate functions, as in Example~\ref{exam:nc-k-ball-functions}, so that $h_1(a) = a_1$ and $h_2(a) = a_2$.

Let $\mu : \rC(K) \to \bC$ denote the state defined by $\mu(f) = \tr_2(f(a))$ for $f \in \rC(K)$. Then $(b,\beta) \in K_4 \times \M_{4,1}$ is a minimal representation for $\mu$, where $b = a \oplus a$ and 
\[
\beta = \left(\frac{1}{\sqrt{2}},0,0,\frac{1}{\sqrt{2}}\right)^T.
\]

Let $F : K \to \M$ be the multivalued nc function defined by
\[
F(x) = \{\alpha \in \M_n : \alpha \geq a_1(x) \text{ and } \alpha \geq a_2(x)\}
\]
for $n \leq \infty$ and $x \in K_n$. Then by definition,
\[
\mu(F) = \beta^* F(b) \beta.
\]

We claim that $\tr_2(F(a)) \ne \mu(F)$. To see this, suppose $\alpha = [\alpha_{i,j}] \in \M_2$ satisfies $\alpha \in F(a)$. Then $\alpha \geq a_1$ and $\alpha \geq a_2$, so $\alpha_{1,1} \geq 1$ and $\alpha_{2,2} \geq 1/2$ respectively.  Note if $\alpha_{1,1} = 1$, then $\alpha \geq a_1$ implies that $\alpha_{1,2} = \alpha_{2,1} = 0$, and thus $\alpha \geq a_2$ implies $\alpha_{2,2} > 1/2$. Hence $\tr_2(\alpha) > 3/4$, so $\tr(F(a)) \subseteq (3/4,\infty)$. 

On the other hand, define $\gamma_1,\gamma_2 \in \M_2$ by $\gamma_1 = 1_2$ and $\gamma_2 = a_1 + a_2$. Then $\gamma_1,\gamma_2 \geq a_1,a_2$, so $\gamma_1,\gamma_2 \in F(a)$. Since $F$ is a multivalued nc function and $a_1 \oplus a_2 = b$, $F(a_1) \oplus F(a_2) \subseteq F(b)$. Hence $\gamma_1 \oplus \gamma_2 \in F(b)$, giving
\[
3/4 = \beta^* (\gamma_1 \oplus \gamma_2) \beta \in \beta^* F(b) \beta = \mu(F).
\]
Therefore, $\tr_2(F(a)) \ne \mu(F)$.
\end{exam}

\section{Noncommutative majorization}  \label{sec:nc-majorization}

In this section, we will define the noncommutative majorization order and prove a characterization of it that completely resolves Problem \ref{intro-problem:motivation}.

\begin{defn}\label{defn:nc-maj-fnc}
Let $K$ be a compact nc convex set.  A \emph{nc majorization function} is a bounded lower semi-continuous convex multivalued nc function $F : K \to \M_p(\M)_{sa}$ for some $p \leq \infty$.
\end{defn}

Let $K$ be a compact nc convex set. For $p \leq \infty$, let $F : K \to \M_p(\M)$ be a multivalued nc function. For $m,n \leq \infty$, let $\psi : \M_n \to \M_m$ be a unital completely positive map. Recall that for $x \in K_n$,
\[
(\id_{\M_p} \otimes \psi)(F(x)) = \{(\id_{\M_p} \otimes \psi)(\alpha) : \alpha \in F(x)\}.
\]
For $x,y \in K_n$, we will write
\[
(\id_{\M_p} \otimes \psi)(F(x)) \leq (\id_{\M_p} \otimes \psi)(F(y))
\]
if for every element $\beta \in (\id_{\M_p} \otimes \psi)(F(y))$, there is an element $\alpha \in (\id_{\M_p} \otimes \psi)(F(x))$ such that $\alpha \leq \beta$. It follows from Definition \ref{defn:multivalued-nc-function} that this is equivalent to the inclusion
\[
(\id_{\M_p} \otimes \psi)(F(x)) \supseteq (\id_{\M_p} \otimes \psi)(F(y)).
\]

\begin{defn} \label{defn:nc-majorization}
For $m,n \in \bN$, let $\psi : \M_n \to \M_m$ be a faithful unital completely positive map. For $k \in \bN$ and $k$-tuples of self-adjoint matrices $a = (a_1,\ldots,a_k), b = (b_1,\ldots,b_k) \in \M_n^k$, we say that $a$ is {\em $\psi$-majorized} by $b$ and write $a \prec_\psi b$ if
\[
(\id_{\M_p} \otimes \psi)(F(a)) \leq (\id_{\M_p} \otimes \psi)(F(b))
\]
for every $p \leq \infty$ and every nc majorization function $F : K \to \M_p(\M)$, where $K$ is any closed nc $k$-ball containing $a$ and $b$. If $\psi = \tr_n$, where $\tr_n : \M_n \to \bC$ is the normalized trace, then we will simply say that $a$ is {\em majorized} by $b$ and write $a \prec b$. 
\end{defn}

\begin{thm} \label{thm:main}
For $m,n \in \bN$, let $\psi : \M_n \to \M_m$ be a faithful unital completely positive map. For $k \in \bN$, let $a = (a_1,\ldots,a_k), b = (b_1,\ldots,b_k) \in \M_n^k$ be $k$-tuples of self-adjoint matrices, and let $\mathcal{B} \subseteq \M_n$ denote the unital C*-algebra generated by $b_1,\ldots,b_k$. The following are equivalent:
\begin{enumerate}
\item $a \prec_\psi b$;
\item There is a unital completely positive map $\Phi : \mathcal{B} \to \M_n$ such that $\psi \circ \Phi = \psi$ and $\Phi(b_i) = a_i$ for all $1 \leq i \leq k$.
\end{enumerate}
\end{thm}

The following corollary is almost immediate.

\begin{cor} \label{cor:bistochastic}
For $k,n \in \bN$ and $k$-tuples of self-adjoint matrices $a = (a_1,\ldots,a_k), b = (b_1,\ldots,b_k) \in \M_n^k$, the following are equivalent:
\begin{enumerate}
\item $a \prec b$;
\item There is a trace-preserving unital completely positive map, i.e. a bistochastic quantum channel, $\Phi : \M_n \to \M_n$ such that $\Phi(b_i) = a_i$ for all $1 \leq i \leq k$.
\end{enumerate}
\end{cor}

\begin{proof}
Let $\mathcal{B} \subseteq \M_n$ denote the unital C*-algebra generated by $b_1,\ldots,b_k$. Apply Theorem \ref{thm:main} with $\psi = \tr_n$ to obtain a trace-preserving unital completely positive map $\Phi' : \mathcal{B} \to \M_n$ satisfying $\Phi'(b_i) = a_i$ for all $1 \leq i \leq k$. Let $E : \M_n \to \mathcal{B}$ be a trace-preserving conditional expectation onto $\mathcal{B}$ and take $\Phi = \Phi' \circ E$.
\end{proof}

Theorem \ref{thm:main} will follow from a more general result, which will be utilized in Section \ref{sec:quantum-channel-interpolation}, when we extend Corollary \ref{cor:bistochastic} to the non-unital setting. Before stating the result, it will be convenient to introduce the following more general definition of nc majorization.

\begin{defn} \label{defn:extended-nc-majorization}
For $m,n_1,n_2 \in \bN$, let $\psi_1 : \M_{n_1} \to \M_m$ and $\psi_2 : \M_{n_2} \to \M_m$ be unital completely positive maps with $\psi_1$ faithful. For $k \in \bN$ and $k$-tuples of self-adjoint matrices $a = (a_1,\ldots,a_k) \in \M_{n_1}^k$ and $b = (b_1,\ldots,b_k) \in \M_{n_2}^k$, we will say that $a$ is $(\psi_1,\psi_2)$-majorized by $b$ and write $a \prec_{\psi_1,\psi_2} b$ if
\[
(\id_{\M_p} \otimes \psi_1)(F(a)) \leq (\id_{\M_p} \otimes \psi_2)(F(b))
\]
for every $p \leq \infty$ and every nc majorization function $F : K \to \M_p(\M)$, where $K$ is any closed nc $k$-ball containing $a$ and $b$.
\end{defn}

\begin{thm} \label{thm:key}
For $m,n_1,n_2 \in \bN$, let $\psi_1 : \M_{n_1} \to \M_m$ and $\psi_2 : \M_{n_2} \to \M_m$ be unital completely positive maps with $\psi_1$ faithful. For $k \in \bN$, let $a = (a_1,\ldots,a_k) \in \M_{n_1}^k$ and $b = (b_1,\ldots,b_k) \in \M_{n_2}^k$ be $k$-tuples of self-adjoint matrices and let $\mathcal{B} \subseteq \M_{n_2}$ denote the unital C*-algebra generated by $b_1,\ldots,b_k$. The following are equivalent:
\begin{enumerate}
\item $a \prec_{\psi_1,\psi_2} b$;
\item There is a unital completely positive map $\Phi : \mathcal{B} \to \M_{n_1}$ such that $\psi_1 \circ \Phi = \psi_2$ and $\Phi(b_i) = a_i$ for all $1 \leq i \leq k$.
\end{enumerate}
\end{thm}

\begin{proof}
(2) $\Rightarrow$ (1) For $p \leq \infty$, let $F : K \to \M_p(\M)$ be an nc majorization function. By \cite{DK2019}*{Theorem 7.4.3}, it suffices to prove that for $q \leq \infty$ and self-adjoint $h \in \M_q(\M_p(\rA(K)))$ with $h \leq 1_q \otimes F$, whenever $\beta \in (\M_p(\M_{n_2}))_{\sa}$ satisfies $1_q \otimes \beta \geq h(b)$, then there is an $\alpha \in (\M_p(\M_{n_1}))_{\sa}$ satisfying $1_q \otimes \alpha \geq h(a)$ and $(\id_{\M_p} \otimes \psi_1)(\alpha) \leq (\id_{\M_p} \otimes \psi_2)(\beta)$. 

For $h$ and $\beta$ as above, let $\alpha = (\id_{\M_p} \otimes \Phi)(\beta)$. Then
\begin{align*}
1_q \otimes \alpha &= (\id_{\M_q} \otimes \id_{\M_p} \otimes \Phi)(1_q \otimes \beta) \\
&\geq (\id_{\M_q} \otimes \id_{\M_p} \otimes \Phi)(h(b)) \\
&= h(\Phi(b)) \\
&= h(a),
\end{align*}
and
\begin{align*}
(\id_{\M_p} \otimes \psi_1)(\alpha) &= (\id_{\M_p} \otimes \psi_1)((\id_{\M_p} \otimes \Phi)(\beta)) \\
&= (\id_{\M_p} \otimes (\psi_1 \circ \Phi))(\beta) \\
&= (\id_{\M_p} \otimes \psi_2)(\beta),
\end{align*}
as desired.

(1) $\Rightarrow$ (2)
Note that $\mathcal{B}$ is the image of $\rC(K)$ under the *-homomorphism $\delta_b$. Let $\mathcal{U} \subseteq \M_{n_2}$ be a complete set of matrix units for $\mathcal{B}$ and let $c = \sum_{e \in \mathcal{U}} e \otimes e \in \M_{n_2} \otimes \M_{n_2}$ denote the corresponding Choi matrix. In other words, $c$ is the direct sum of the Choi matrices for each of the full matrix summands of $\mathcal{B}$.

Since $c$ is in the range of the point evaluation *-homomorphism $\id_{\M_{n_2}} \otimes \delta_b : \M_{n_2}(\rC(K)) \to \M_{n_2} \otimes \M_{n_2}$ and $c\geq 0$, it is a standard fact from the theory of C*-algebras that there is $g \in \M_{n_2}(\rC(K))$ with $g \geq 0$ such that $(\id_{\M_{n_2}} \otimes \delta_b)(g) = c$. Note that $g$ can be written as $g = \sum_{e \in \mathcal{U}} e \otimes g_e$ for $g_e \in \rC(K)$ for each $e \in \mathcal{U}$. Let $\mathcal{U}_d \subseteq \mathcal{U}$ be the set of diagonal matrix units and define self-adjoint $h \in \rC(K)$ by
\[
h = 1_{\rC(K)} - \sum_{e \in \mathcal{U}_d} g_e.
\]

Let $p = 2|\mathcal{U}| + 2$ and define $f \in \M_p(\rC(K))$ by
\[
f = \left( \bigoplus_{e \in \mathcal{U}} (g_e \oplus (-g_e)) \right) \oplus h^*h \oplus (-h^*h).
\]
Let $F := \overline{f}$ denote the convex envelope of $f$. By \cite{DK2019}*{Proposition 7.4.2}, $F$ is an nc majorization function satisfying $F \leq f$. Hence by assumption,
\begin{equation} \label{eq:key-1}
(\id_{\M_p} \otimes \psi_1)(F(a)) \leq (\id_{\M_p} \otimes \psi_2)(F(b)) \leq (\id_{\M_p} \otimes \psi_2)(f(b)).
\end{equation}
On the other hand, by \cite{DK2019}*{Theorem 7.5.1},
\[
F(a) = \bigcup_\phi [(\id_{\M_p} \otimes \phi)(f),\infty),
\]
where the union is taken over all unital completely positive maps $\phi : \rC(K) \to \M_{n_1}$ with barycentre $a$. Applying $\psi_1$ to this equation gives
\begin{align} \label{eq:key-2}
\begin{split}
&(\id_{\M_p} \otimes \psi_1)(F(a)) \\
&\quad = \bigcup_\phi \{ (\id_{\M_p} \otimes \psi_1)(\alpha) : \alpha \in (\M_p(\M_{n_1}))_{\sa},\ \alpha \geq (\id_{\M_p} \otimes \phi)(f) \}.
\end{split}
\end{align}
Together, (\ref{eq:key-1}) and (\ref{eq:key-2}) imply
\begin{align*}
(\id_{\M_p} &\otimes \psi_2)(f(b)) \\
& \in \bigcup_\phi \{ (\id_{\M_p} \otimes \psi_1)(\alpha) : \alpha \in (\M_p(\M_{n_1}))_{\sa},\ \alpha \geq (\id_{\M_p} \otimes \phi)(f) \}.
\end{align*}
Hence there is a unital completely positive map $\phi : \rC(K) \to \M_{n_1}$ with barycentre $a$ and $\alpha \in \M_p(\M_{n_1})$ with $\alpha \geq (\id_{\M_p} \otimes \phi)(f)$ such that
\begin{equation} \label{eq:key-3}
(\id_{\M_p} \otimes \psi_2)(f(b)) = (\id_{\M_p} \otimes \psi_1)(\alpha) \geq (\id_{\M_p} \otimes \psi_1 \circ \phi)(f).
\end{equation}

From (\ref{eq:key-3}) and the definition of $f$, we obtain
\[
(\psi_1 \circ \phi)(g_e) \leq \psi_2(g_e(b)), \quad (\psi_1 \circ \phi)(-g_e) \leq \psi_2(-g_e(b))
\]
for all $e \in \mathcal{U}$, and
\[
(\psi_1 \circ \phi)(h^*h) \leq \psi_2(h^*(b)h(b)), \quad (\psi_1 \circ \phi)(-h^*h) \leq \psi_2(-h^*(b)h(b)).
\]
Hence
\begin{equation} \label{eq:key-4}
(\psi_1 \circ \phi)(g_e) = \psi_2(g_e(b))
\end{equation}
for all $e \in \mathcal{U}$, and
\begin{equation} \label{eq:key-5}
(\psi_1 \circ \phi)(h^*h) = \psi_2(h^*(b)h(b)).
\end{equation}

Observe that
\[
h(b) = 1_{n_2} - \sum_{e \in \mathcal{U}_d} e = 0,
\]
so applying the Cauchy-Schwarz inequality to (\ref{eq:key-5}) gives
\[
0 \leq \psi_1(\phi(h^*)\phi(h)) \leq (\psi_1 \circ \phi)(h^*h) = \psi_2(h^*(b)h(b)) = 0.
\]
It follows from the faithfulness of $\psi_1$ that $\phi(h) = 0$. Hence
\begin{equation} \label{eq:key-6}
\sum_{e \in \mathcal{U}_d} \phi(g_e) = \phi(1_{\rC(K)}) = 1_{n_1}.
\end{equation}

Define $\Phi : \mathcal{B} \to \M_{n_1}$ to be the unique linear map satisfying
\[
\Phi(e) = \phi(g_e)
\]
for all $e \in \mathcal{U}$. Then (\ref{eq:key-4}) implies
\[
\psi_1 \circ \Phi = \psi_2.
\]
We claim that $\Phi$ is unital and completely positive. To see that $\Phi$ is unital, we use (\ref{eq:key-6}) to compute
\[
\Phi(1_{n_2}) = \sum_{e \in \mathcal{U}_d} \Phi(e) = \sum_{e \in \mathcal{U}_d} \phi(g_e) = 1_{n_1}.
\]
To see that $\Phi$ is completely positive, we observe that the image of the Choi matrix $c$ under $\id_{\M_{n_2}} \otimes \Phi$ satisfies
\begin{align*}
(\id_{\M_{n_2}} \otimes \Phi)(c) &= \sum_{e \in \mathcal{U}} e \otimes \Phi(e) \\
&= \sum_{e \in \mathcal{U}} e \otimes \phi(g_e) \\
&= (\id_{\M_{n_2}} \otimes \phi)(g) \\
&\geq 0,
\end{align*}
since $g \geq 0$ and $\phi$ is completely positive.
\end{proof}

\begin{rem}
It follows from the proof that conditions (1) and (2) of Theorem \ref{thm:key} are equivalent to the weaker condition that
\[
(\id_{\M_p} \otimes \psi_1)(F(a)) \leq (\id_{\M_p} \otimes \psi_2)(F(b))
\]
for every $p \in \bN$ with $p \leq 2n_2^2 + 2$ and every nc majorization function $F : K \to \M_p(\M)$. We will say more about this in Section \ref{sec:effective}.
\end{rem}

\section{Interpolation by quantum channels}  \label{sec:quantum-channel-interpolation}

In this section, we will extend Corollary \ref{cor:bistochastic} to the case of trace-preserving completely positive maps, i.e. quantum channels that are not necessarily unital. The statement of the following result utilizes Definition~\ref{defn:extended-nc-majorization}.

\begin{thm} \label{thm:non-unital}
For $k,n_1,n_2 \in \bN$ and $k$-tuples of self-adjoint matrices $a = (a_1,\ldots,a_k) \in \M_{n_1}^k$ and $b = (b_1,\ldots,b_k) \in \M_{n_2}^k$, the following are equivalent:
\begin{enumerate}
\item $a' \prec_{\tr_{n_1},\omega} b'$, where $a' \in \M_{n_1}^k$ and $b' \in \M_{n_2 + 1}$ are defined by
\[
a' = \frac{1}{n_1} a, \quad b' = (b_1 \oplus 0, \ldots, b_k \oplus 0),
\]
and $\omega : \M_{n_2 + 1} \to \bC$ is the state defined by
\[
\omega([x_{ij}]) = \frac{1}{n_1} \sum_{i=1}^{n_2} x_{ii} + \frac{n_1 - 1}{n_1} x_{n_2 + 1, n_2 + 1};
\]
\item There is a trace-preserving completely positive map, i.e. a quantum channel, $\Phi : \M_{n_2} \to \M_{n_1}$ such that $\Phi(b_i) = a_i$ for all $1 \leq i \leq k$. 
\end{enumerate}
\end{thm}

\begin{proof}
Define $\psi_1 : \M_{n_1} \to \bC$ and $\psi_2 : \M_{n_2} \to \bC$ by
\[
\psi_1 = \tr_{n_1},\quad \psi_2 = \frac{1}{n_1} \tr_{n_2}
\]
We first claim that the existence of a trace-preserving completely positive map $\Phi : \M_{n_2} \to \M_{n_1}$ such that $\Phi(b_i) = a_i$ for all $1 \leq i \leq k$ is equivalent to the existence of a contractive completely positive map $\phi : \M_{n_2} \to \M_{n_1}$ such that $\psi_1 \circ \phi = \psi_2$ and $\phi(b_i) = \frac{1}{n_1} a_i$ for all $1 \leq i \leq k$. 

To see the reverse direction of this claim, observe given $\phi$ as above, the completely positive map $\Phi : \M_{n_2} \to \M_{n_1}$ defined by $\Phi = n_1 \phi$ is trace-preserving and satisfies $\Phi(b_i) = a_i$ for all $1 \leq i \leq k$. Conversely, given $\Phi$ as above, the completely positive map $\phi : \M_{n_2} \to \M_{n_1}$ defined by $\phi = \frac{1}{n_1} \Phi$ satisfies
\[
\psi_1 \circ \phi = \frac{1}{n_1} \tr_{n_1} \circ \Phi = \frac{1}{n_1} \tr_{n_2} = \psi_2,
\]
and $\phi(b_i) = \frac{1}{n_1} a_i$ for all $1 \leq i \leq k$. Furthermore, $\phi$ is completely contractive, since
\[
\|\phi(1_{n_1})\| = \frac{1}{n_1} \|\Phi(1_{n_1})\| \leq \|\Phi(1_{n_1})\|_2 \leq 1,
\]
establishing the claim.

Let $E_1 : \M_{n_2 + 1} \to \M_{n_2}$ denote the map compressing $\M_{n_2+1}$ to the top-left corner of size $n_2$ and let $E_2 : \M_{n_2 + 1} \to \bC$ denote the map compressing $\M_{n_2 + 1}$ to the bottom-right corner of size $1$. Define a state $\omega : \M_{n_2 + 1} \to \bC$ by
\[
\omega = \frac{1}{n_1} \tr_{n_2} \circ E_1 + \frac{n_1 - 1}{n_1} E_2.
\]

Define $a' \in \M_{n_1}^k$ by $a' = \frac{1}{n_1}a$, and define $b' \in \M_{n_2 + 1}^k$ by $b' = (b_1 \oplus 0, \ldots, b_k \oplus 0)$. We next claim that the existence of the map $\phi$ as above is equivalent to the existence of a unital completely positive map $\Phi' : \M_{n_2 + 1} \to \M_{n_1}$ such that $\psi_1 \circ \Phi' = \omega$ and $\Phi'(b'_i) = a'_i$ for all $1 \leq i \leq k$.

To see the forward direction of this claim, assume $\phi$ is given as above. Define $\Phi' : \M_{n_2 + 1} \to \M_{n_1}$ by
\[
\Phi' = \phi \circ E_1 + (1_{n_1} - \phi(1_{n_2})) \cdot E_2,
\]
where $(1_{n_1} - \phi(1_{n_2})) \cdot E_2$ denotes multiplication of $(1_{n_1} - \phi(1_{n_2}))$ by the scalar obtained by applying $E_2$. Then $\Phi'$ is unital and completely positive, $\psi_1 \circ \Phi' = \omega$ and $\Phi'(b'_i) = a'_i$ for all $1 \leq i \leq k$. 

Conversely, given $\Phi'$ as above, let $P : \M_{n_2} \to \M_{n_2 + 1}$ denote the map embedding $\M_{n_2}$ into the top-left corner of $\M_{n_2+1}$. Then the contractive completely positive map $\phi : \M_{n_2} \to \M_{n_1}$ defined by $\phi = \Phi' \circ P$ satisfies $\phi(b_i) = \frac{1}{n_1} a_i$ for all $1 \leq i \leq k$, establishing the claim.

We may now apply Theorem \ref{thm:key} with $\psi_1 = \tr_{n_1}$ and $\omega$, which asserts that a map $\Phi'$ as above exists if and only if
\[
(\id_{\M_p} \otimes \tr_{n_1})(F(a')) \leq (\id_{\M_p} \otimes \omega)(F(b'))
\]
for every $p \leq 2(n_2+1)^2+2$ and every nc majorization function $F : K \to \M_p(\M)$, i.e. if and only if $a' \prec_{\tr_{n_1},\omega} b'$.
\end{proof}

\section{Effective verification of noncommutative majorization}  \label{sec:effective}

The definition of the nc majorization orders in Definition \ref{defn:nc-majorization} requires testing against nc majorization functions which, by \cite{DK2019}*{Theorem 7.4.3}, arise as supremums of typically infinite families of continuous affine nc functions. In this section, we will prove that the nc majorization orders can instead be tested against multivalued nc functions obtained as supremums of finitely many continuous affine nc functions and thus a single one by Example \ref{exam:multiple-affine-to-a-single}. These functions are particularly tractable, although we note that they are not necessarily bounded in the sense of Definition \ref{defn:multivalued-nc-function}.

\begin{thm}
For $m,n_1,n_2 \in \bN$, let $\psi_1 : \M_{n_1} \to \M_m$ and $\psi_2 : \M_{n_2} \to \M_m$ be unital completely positive maps with $\psi_1$ faithful. For $k \in \bN$, and $k$-tuples of self-adjoint matrices $a = (a_1,\ldots,a_k) \in \M_{n_1}^k$ and $b = (b_1,\ldots,b_k) \in \M_{n_2}^k$, the following are equivalent:
\begin{enumerate}
\item $a \prec_{\psi_1,\psi_2} b$;
\item For every $p \in \bN$ with $p \leq 2n^2_2 + 2$, every $q,r_1,\ldots,r_q \in \bN$ and $h_i \in \M_{r_i}(\M_p(\rA(K)))$ for $1 \leq i \leq q$, if $G : K \to \M_p(\M)_{\operatorname{sa}}$ is the multivalued nc function defined by
\[
G(x) = \bigcap_{i=1}^q \{\alpha \in \M_p(\M_n)_{\operatorname{sa}} : 1_{r_i} \otimes \alpha \geq h_i(x)\},
\]
then
\[
(\id_{\M_p} \otimes \psi_1)(G(a)) \leq (\id_{\M_p} \otimes \psi_2)(G(b)).
\]
Here, $K$ is any closed nc $k$-ball containing $a$ and $b$.
\end{enumerate}
\end{thm}

\begin{proof}
(1) $\Rightarrow$ (2)
Let $G$ be a function as defined in the statement of the theorem. Then letting $\mathcal{B} \subseteq \M_{n_2}$ denote the unital C*-algebra generated by $b_1,\ldots,b_k$, Theorem \ref{thm:key} implies that there is a unital completely positive map $\Phi : \mathcal{B} \to \M_{n_1}$ such that $\psi \circ \Phi = \psi_2$ and $\Phi(b_i) = a_i$ for all $1 \leq i \leq k$. 

For $\alpha \in G(b)$, $1_{r_i} \otimes \alpha \geq h_i(b)$ for all $1 \leq i \leq q$. Hence
\begin{align*}
1_{r_i} \otimes (\id_{\M_p} \otimes \Phi)(\alpha) &= (\id_{\M_{r_i}} \otimes \id_{\M_p} \otimes \Phi)(1_{r_i} \otimes \alpha) \\
&\geq (\id_{\M_{r_i}} \otimes \id_{\M_p} \otimes \Phi)(h_i(b)) \\
&= h_i(a).
\end{align*}
Thus
\[
(\id_{\M_p} \otimes \Phi)(\alpha) \in G(a).
\]
Moreover, since $\psi_1 \circ \Phi = \psi_2$,
\[
(\id_{\M_p} \otimes \psi_1)((\id_{\M_p} \otimes \Phi)(\alpha)) = (\id_{\M_p} \otimes \psi_2)(\alpha).
\]
Hence
\[
(\id_{\M_p} \otimes \psi_1)(G(a)) \leq (\id_{\M_p} \otimes \psi_2)(G(b)),
\]
as required.

(2) $\Rightarrow$ (1)
For $p \in \bN$, let $F : K \to \M_p(\M)$ be an nc majorization function. Since $F$ is bounded, by translating we can assume that $F \geq 0$.

Let $\epsilon > 0$ and, for $\lambda \geq 0$, let
\[
B_\lambda = \{\beta \in (\M_p(\M_m))_{\sa} : \|\beta\| \leq \lambda\}.
\]
We claim that it suffices to prove that there is a multivalued nc function $G$, defined as in the statement of the theorem, such that $G \leq F$ and, for sufficiently large $\lambda \geq 0$,
\begin{equation} \label{eq:affine-1}
((\id_{\M_p} \otimes \psi_1)(F(a)) - \epsilon 1_{pm}) \cap B_\lambda \supseteq \left((\id_{\M_p} \otimes \psi_1)(G(a))\right) \cap B_\lambda.
\end{equation}
Indeed, suppose that this is true. Then since $G \leq F$,
\begin{align*}
((\id_{\M_p} \otimes \psi_1)(F(a))- \epsilon 1_{pm}) \cap B_\lambda &\supseteq ((\id_{\M_p} \otimes \psi_1)(G(a))) \cap B_\lambda  \\
& \supseteq ((\id_{\M_p} \otimes \psi_2)(G(b)))  \cap B_\lambda \\
& \supseteq ((\id_{\M_p} \otimes \psi_2)(F(b)))  \cap B_\lambda.
\end{align*}
Then taking $\lambda \to \infty$ gives
\[
(\id_{\M_p} \otimes \psi_1)(F(a))- \epsilon 1_{pm} \supseteq (\id_{\M_p} \otimes \psi_2)(F(b)).
\]
Since $F$ is lower semicontinuous, $F(a)$ is closed. Thus $(\id_{\M_p} \otimes \psi_1)(F(a))$ is closed. Since $\epsilon$ was arbitrary, it follows that we can also take $\epsilon \to 0$, giving
\[
(\id_{\M_p} \otimes \psi_1)(F(a)) \supseteq (\id_{\M_p} \otimes \psi_2)(F(b)).
\]
Hence
\[
(\id_{\M_p} \otimes \psi_1)(F(a)) \leq (\id_{\M_p} \otimes \psi_2)(F(b)),
\]
establishing the claim.

The rest of the proof is devoted to constructing, for fixed $\epsilon > 0$, a multivalued nc function $G \leq F$ such that for sufficiently large $\lambda \geq 0$, (\ref{eq:affine-1}) holds. 

First note that since $\psi_1$ is faithful, it is bounded from below on positive matrices. Hence for every $\lambda > 0$, there is $\lambda' \geq \lambda$ such that whenever $\beta \in B_\lambda$ with $\beta \geq 0$ is of the form $\beta = (\id_{\M_p} \otimes \psi_1)(\gamma)$ for some $\gamma \in \M_p(\M_m)_{\sa}$ with $\gamma \geq 0$, then $\|\gamma\| \leq \lambda'$.

Define $C \subseteq \M_p(\M_{n_1})^+$ by
\[
C = \{\gamma \in \M_p(\M_{n_1})^+ : \|\gamma\| \leq \lambda' \text{ and } \gamma + \epsilon 1_{p n_1} \notin F(a) \}.
\]
Clearly $C$ is a bounded set of positive matrices, and in particular is precompact. Hence there is $q \in \bN$ and $\gamma_1,\ldots, \gamma_q \in C$ such that $\{\gamma_1,\ldots, \gamma_q\}$ is an $\epsilon$-net for $C$.

By \cite{DK2019}*{Theorem 7.4.3},
\[
F(a) = \bigcap_r \bigcap_{h \leq 1_r \otimes F} \left\{ \alpha \in \M_p(\M_m)_{\sa} : 1_r \otimes \alpha \geq h(a)\right\}
\]
where the intersection is taken over all $r$ and all self-adjoint nc affine functions $h \in \M_r(\M_p(\rA(K)))$ satisfying $h \leq 1_r \otimes F$. For each $1 \leq i \leq q$, since $\gamma_i + \epsilon 1_{p n_1}\notin F(a)$, there is $r_i \in \bN$ and $h_i \in \M_{r_i}(\M_p(\rA(K)))$ such that
\[
\gamma_i + \epsilon 1_{p n_1} \notin \{\alpha \in \M_p(\M_{n_1})_{\sa} : 1_{r_i} \otimes \alpha \geq h_i(a)\}.
\]
Let $r_{q+1} = 1$ and let $h_{q+1} \in \M_p(A(K))$ denote the zero affine function. 

Define the multivalued nc function $G : K \to \M_p(\M)$ by
\[
G(x) = \bigcap_{i=1}^{q+1} \{\alpha \in \M_p(\M_n)_{sa} : 1_{r_i} \otimes \alpha \geq h_i(x)\},
\]
for $n \leq \infty$ and $x \in K_n$.  Then $G \geq 0$ and $G \leq F$ by construction. We claim that (\ref{eq:affine-1}) holds, i.e. that
\[
((\id_{\M_p} \otimes \psi_1)(F(a)) - \epsilon 1_{pm}) \cap B_\lambda \supseteq ((\id_{\M_p} \otimes \psi_1)(G(a))) \cap B_\lambda.
\]
To see this, let $\beta \in ((\id_{\M_p} \otimes \psi_1)(G(a))) \cap B_\lambda$. Since $G \geq 0$, $\beta \geq 0$, so from above there is $\gamma \in G(a)$ such that $\gamma \geq 0$, $\|\gamma\| \leq \lambda'$ and $\beta = (\id_{\M_p} \otimes \psi_1)(\gamma)$.

We claim $\gamma \notin C$. Otherwise there would be $1 \leq i \leq q$ such that $\|\gamma - \gamma_i\| < \epsilon$. Then setting $\gamma' = \gamma - \gamma_i$ yields $\|\gamma'\| < \epsilon$, $1_{p n_1} - \gamma' \geq 0$, and $\gamma_i + \gamma' \in G(a)$. Since $G(a)$ is directed upwards, this would imply
\[
\gamma_i + \epsilon 1_{p n_1} = \gamma_i + \gamma' + (\epsilon 1_{p n_1} - \gamma') \in G(a),
\]
contradicting the construction of $G$.

Since $\gamma \notin C$, it follows that $\gamma \in F(a) - \epsilon 1_{p n_1}$. Hence there is $\gamma_0 \in F(a)$ such that $\gamma = \gamma_0 - \epsilon 1_{p n_1}$. Then
\begin{align*}
\beta &= (\id_{\M_p} \otimes \psi_1)(\gamma) \\
&= (\id_{\M_p} \otimes \psi_1)(\gamma_0 - \epsilon 1_{p n_1}) \\
&\in (\id_{\M_p} \otimes \psi_1)(F(a)) - \epsilon 1_{pm},
\end{align*}
completing the proof.
\end{proof}

\section{Necessity of multivalued noncommutative functions}  \label{sec:necessity}

It is natural to ask whether higher order multivalued nc functions are truly required in the definition of the nc majorization order, or whether they are simply an artifact of our proofs. It would certainly be preferable if it was instead possible to consider only first order continuous convex nc functions. However, as we will demonstrate in this section, the consideration of higher order multivalued nc functions is indeed necessary, in a very strong sense.

The following example demonstrates the necessity of considering multivalued nc functions.

\begin{prop}
There is $n \in \bN$ and $3$-tuples of commuting self-adjoint matrices $a = (a_1,a_2,a_3), b = (b_1,b_2,b_3) \in \M_n^3$ such that
\[
(\id_{\M_p} \otimes \tr_n)(f(a)) \leq (\id_{\M_p} \otimes \tr_n)(f(b))
\]
for every $p \leq \infty$ and every continuous convex nc function $f \in \M_p(\rC(K))$,
but there does not exist a trace-preserving unital completely positive map $\Phi : \M_n \to \M_n$ such that $\Phi(b_i) = a_i$ for all $1 \leq i \leq k$.
\end{prop}
\begin{proof}
Let $n = 15$. Let $\{e_1, e_2, e_3\}$ be the standard basis for $\bC^3$ and define $y,x_1,x_2,x_3 \in \bC^3$ by
\[
y = \frac{1}{3} \sum^3_{k=1} e_k, \quad x_i = \frac{1}{4}\left(e_i + \sum^3_{j=1} e_j\right)
\]
for $1 \leq i \leq 3$. Let $a_1,a_2,a_3 \in \M_n$ be diagonal matrices with joint spectral measure
\[
\mu = \frac{1}{15} \sum^3_{i=1} 5 \delta_{x_i},
\]
i.e. each $x_i$ appears as a joint eigenvalue with multiplicity $5$. Similarly, let $b_1,b_2,b_3 \in \M_n^3$ be diagonal matrices with joint spectral measure
\[
\nu = \frac{12}{15} \delta_y + \frac{1}{15} \sum^3_{i=1} \delta_{x_i}
\]
i.e. $y$ appears as a joint eigenvalue with multiplicity $12$ and each $x_i$ appears as a joint eigenvalue with multiplicity $1$.

It follows from \cite{DK2021}*{Theorem 9.2} and \cite{DK2019}*{Theorem 8.5.1} that
\[
(\id_{\M_p} \otimes \tr_n)(f(a)) \leq (\id_{\M_p} \otimes \tr_n)(f(b)),
\]
for every $p \leq \infty$ and every continuous convex nc function $f \in \M_p(C(K))$. However, it follows from \cite{DK2021}*{Theorem 9.2} that $\mu$ is not dominated by $\nu$ in the classical Choquet order on probability measures, so the results of \cite{MMS2005} imply that there does not exist a trace-preserving unital completely positive map $\Phi : \M_n \to \M_n$ such that $\Phi(b_i) = a_i$ for all $1 \leq i \leq 3$.
\end{proof}

The following example demonstrates the necessity of considering higher order multivalued nc functions.

\begin{prop}
There exists $k,n \in \bN$ and self-adjoint $k$-tuples $a = (a_1,\ldots, a_k), b = (b_1,\ldots,b_k) \in \M_n^k$ such that for every $l \in \bN$ and $h_1,\ldots,h_l \in \rA(K)$, if $G : K \to \M$ is the multivalued nc function defined by
\[
G(x) = \bigcap_{i=1}^l \{\alpha \in (\M_m)_{\sa} : \alpha \geq a_i(x)\}
\]
for $m \in \bN$ and $x \in K_m$, then
\[
\tr_n(G(a)) \leq \tr_n(G(b)),
\]
but there does not exist a unital quantum channel $\Phi : \M_n \to \M_n$ such that $\Phi(b_i) = a_i$ for all $1 \leq i \leq k$. Here $K$ is any closed nc $k$-ball containing $a$ and $b$.
\end{prop}
\begin{proof}
Let $b_1,\ldots, b_4$ be a self-adjoint basis of $\M_2$, and let $a_k = b_k^T$ (the transpose of $b_k)$ for all $1 \leq i \leq 4$. Then $a = (a_1,a_2,a_3,a_4), b = (b_1,b_2,b_3,b_4) \in \M_2^4$ are self-adjoint tuples of matrices. If $\Phi : \M_n \to \M_n$ is a linear map such that $\Phi(b_i) = a_i$ for all $1 \leq i \leq 4$, then $\Phi$ must be the transpose map, which is not completely positive. Hence there is no trace-preserving unital completely positive map $\Phi : \M_2 \to \M_2$ such that $\Phi(b_i) = a_i$ for all $1 \leq i \leq 4$. 

However, it follows from Example \ref{exam:nc-k-ball-functions} that for $h \in \rA(K)$ and $\alpha \in (\M_n)_{\sa}$, $\alpha \geq h(a)$ if and only if $\alpha^T \geq h(b)$. Since the transpose map is trace-preserving, it follows that
\[
\tr(G(a)) = \tr(G(b))
\]
for all multivalued nc functions $G$ as above.
\end{proof}

\section{Mixed unitary quantum channels} \label{sec:mixed}

For $n \in \bN$, a unital completely positive map $\Phi : \M_n \to \M_n$ is said to be {\em mixed unitary} if there is $k \in \bN$, unitaries $u_1,\ldots,u_k \in \M_n$, and $t_1,\ldots,t_k \in [0,1]$ with $\sum_{i=1}^k t_i = 1$ such that $\Phi$ is of the form
\[
\Phi(x) = \frac{1}{k} \sum_{i=1}^k u_i x u_i^*, \quad \text{for all }x \in \M_n.
\]
Note that a mixed unitary completely positive map is unital and trace-preserving. In fact, mixed unitary maps are, in some sense, the ``ideal'' examples of such maps.

Tregrub \cite{T1986} showed that for $n = 2$, every trace-preserving unital completely positive map is mixed unitary. However, he also showed that for every $n \geq 3$, there are trace-preserving unital completely positive maps that are not mixed unitary. This stimulated a great deal of work on the topic of bistochastic quantum channels, culminating in the disproof by Haagerup and Musat \cite{HM2011} of the asymptotic quantum Birkhoff conjecture.

Corollary \ref{cor:bistochastic} asserts that for $k \in \bN$ and $k$-tuples of self-adjoint matrices $a = (a_1,\ldots,a_k), b = (b_1,\ldots,b_k) \in \M_n^k$ satisfying $a \prec b$, there is a trace-preserving unital completely positive map $\Phi : \M_n \to \M_n$ such that $\Phi(b_i) = a_i$ for all $1 \leq i \leq k$. It is natural to ask if the map $\Phi$ can be chosen to be mixed unitary.

For $n \in \bN$ and $k=1$, it follows easily from Theorem \ref{intro-thm:classical-characterizations} that this question has a positive answer. Also, for $n \leq 2$ and $k \in \bN$, this question has a positive answer by Tregrub's theorem.

On the other hand, it follows from the results of Tregrub that if $n \geq 3$ and $k \geq n^2$, then this question has a negative answer. To see this, let $n \geq 3$ and let $\Phi : \M_n \to \M_n$ be a trace-preserving completely positive map that is not mixed unitary. Let $\{b_1,\ldots,b_{n^2}\}$ be a self-adjoint basis for $\M_n$ and define $a,b \in \M_n^{n^2}$ by
\[
b = (b_1,\ldots,b_{n^2}), \quad a = (\Phi(b_1),\ldots,\Phi(b_{n^2})).
\]
By construction, $a \prec b$ and $\Phi(b_i) = a_i$ for all $1 \leq i \leq k$.  Clearly $\Phi$ is the unique trace-preserving completely positive map with this property and was chosen to be not mixed unitary.

For $n \in \bN$ and $k = 2$, this question has an interesting interpretation. Observe that taking real and imaginary parts implements a bijection between $\M_n$ and self-adjoint tuples in $\M_n^2$. So the question is equivalent to asking if every trace-preserving unital completely positive map on $\M_n$ is ``locally'' mixed unitary, in the following sense.

\begin{defn}
For $n \in \bN$, we will say that a trace-preserving unital completely positive map $\Phi : \M_n \to \M_n$ is {\em locally mixed unitary} if for every $x \in \M_n$ there is $k \in \bN$, unitaries $u_1,\ldots,u_k \in \M_n$, and $t_1,\ldots,t_k \in [0,1]$ with $\sum_{i=1}^k t_i = 1$ such that
\[
\Phi(x) = \frac{1}{k} \sum_{i=1}^k u_i x u_i^*.
\]
\end{defn}

We will show by example that for all $n \geq 4$, there are trace-preserving unital completely positive maps on $\M_n$ that are not locally mixed unitary. Since the examples will be Schur multipliers, we now briefly review some basic facts about these maps.

Recall that for $n \in \bN$ and matrices $a = [a_{ij}], b = [b_{ij}] \in \M_n$, the Schur product of $a$ and $b$ is the matrix $a \ast b \in \M_n$ defined by
\[
a \ast b = [a_{ij} b_{ij}],
\] 

\begin{defn}
For $n \in \bN$ and a matrix $c \in \M_n$, the corresponding Schur multiplier is the linear map $T_c : \M_n \to \M_n$ defined by
\[
T_c(x) = c \ast x, \quad x \in \M_n.
\]
\end{defn}

\begin{rem}
We recall the well known fact that for $c \in \M_n$, the corresponding Schur multiplier $T_c$ is unital and completely positive if and only if $c \geq 0$ and $c_{ii} = 1$ for all $1 \leq i \leq n$.  
\end{rem}

We are grateful to Laurent Marcoux for several helpful conversations on this topic, and for providing us with the following result.

\begin{prop} \label{prop:local-mixed-equals-mixed}
For $n \in \bN$, let $c = [c_{ij}] \in \M_n$ be a matrix with $c \geq 0$ and $c_{ii} = 1$ for all $1 \leq i \leq n$, so that the corresponding Schur multiplier $T_c : \M_n \to \M_n$ is unital and completely positive. If $T_c$ is locally mixed unitary, then it is mixed unitary.
\end{prop}
\begin{proof}
Suppose that $T_c$ is locally mixed unitary. Let $h \in \M_n$ be a diagonal matrix with distinct diagonal entries and let $q \in \M_n$ be a matrix with every entry equal to $1$. Then $T_c(h) = h$ and $T_c(q) = c$. Let $a = h + iq$. 

Since $T_c$ is locally mixed unitary, there is $m \in \bN$, unitaries $u_1,\ldots,u_m \in \M_n$, and $t_1,\ldots,t_m \in [0,1]$ with $\sum_{k=1}^m t_k = 1$ such that
\[
T_c(a) = \sum_{k=1}^m t_k u_k x u_k^*.
\]
Since $T_c$ is positive,
\[
\begin{split}
h &= T_c(\operatorname{Re}(a)) = \operatorname{Re}(T_c(a)) = \sum_{k=1}^m t_k \operatorname{Re}(u_i a u_k^*) \\
&= \sum_{k=1}^m t_k u_k \operatorname{Re}(a) u_k^* = \sum_{k=1}^m t_k u_k h u_k^*,
\end{split}
\]
and it is easy to check that this implies $u_k h = h u_k$ for all $1 \leq k \leq m$. Then since $h$ is a diagonal matrix with distinct diagonal entries, this forces each $u_k$ to also be a diagonal matrix. Hence we can write each $u_k$ as
\[
u_k = \operatorname{diag}(\omega_1^{(k)},\ldots,\omega_n^{(k)})
\]
for $\omega_1^{(kj)},\ldots,\omega_n^{(k)} \in \bT$.

Similarly,
\[
\begin{split}
c &= T_c(\operatorname{Im}(a)) = \operatorname{Im}(T_c(a)) = \sum_{k=1}^m t_k \operatorname{Im}(u_k a u_k^*) \\
&= \sum_{k=1}^m t_k u_k \operatorname{Im}(a) u_k^* = \sum_{k=1}^m t_k u_k q u_k^*,
\end{split}
\]
and it follows from a simple computation that
\[
c_{ij} = \sum_{k=1}^m t_k \omega_i^{(k)} \overline{\omega_j^{(k)}}
\]
for all $1 \leq i,j \leq n$.

Therefore, for arbitrary $x = [x_{ij}] \in \M_n$,
\begin{align*}
T_c(x) &= [c_{ij} x_{ij}] \\
&= \left[ \sum_{k=1}^m t_k \omega_i^{(k)} \overline{\omega_j^{(k)}} x_{ij} \right] \\
&= \sum_{k=1}^m t_k \left[ \omega_i^{(k)} \overline{\omega_j^{(k)}} x_{ij} \right] \\
&= \sum_{k=1}^m t_k u_k^* x u_k,
\end{align*}
showing that $T_c$ is mixed unitary.
\end{proof}

Applying Proposition \ref{prop:local-mixed-equals-mixed} along with the results from \cite{HM2011}, we obtain examples of trace-preserving unital completely positive maps that are not locally mixed unitary.

\begin{exam}\label{exam:no-mixed-unitaries-for-2}
For $0 < s < 1$, let
\[
c(s) = \begin{bmatrix} 1 & \sqrt{s} & \sqrt{s} &  \sqrt{s} \\  \sqrt{s} & s & s &s \\ \sqrt{s} & s & s & s \\  \sqrt{s} & s & s & s \end{bmatrix} + (1-s) \begin{bmatrix} 0 & 0 & 0 & 0 \\ 0 &1 & \omega & \overline{\omega} \\ 0 & \overline{\omega} & 1 & \omega \\ 0 & \omega & \overline{\omega} & 1 \end{bmatrix}, \]
where $\omega = e^{2 \pi i/3}$ is a cube root of unity.  By \cite{HM2011}*{Example 3.2}, $T_{c(s)}$ is not factorizable, and so is not a mixed unitary quantum channel.  However, $c(s)$ is positive so $T_{c(s)}$ is completely positive, and it is easy to verify that $T_{c(s)}$ is trace-preserving and unital.  Thus $T_{c(s)}$ is a trace-preserving unital completely positive map that is not mixed unitary. Furthermore, Proposition \ref{prop:local-mixed-equals-mixed} implies that $T_{c(s)}$ is not locally mixed unitary.

In particular, define $a = (a_1,a_2),\ b = (b_1,b_2) \in \M_4^2$ by
\[
b_1 = \begin{bmatrix} 1 & 0 & 0 & 0 \\ 0 & 2 & 0 & 0 \\ 0 & 0 & 3 & 0 \\ 0 & 0 & 0 & 4 \end{bmatrix}, \quad b_ 2= \begin{bmatrix} 1 & 1 & 1 & 1 \\ 1 & 1 & 1 & 1 \\ 1 & 1 & 1 & 1 \\ 1 & 1 & 1 & 1 \end{bmatrix},
\]
and $a_i = T_{c(s)}(b_i)$ for $1 \leq i \leq 2$. Then it follows from Corollary \ref{cor:bistochastic} that $a \prec b$, and by construction $\Phi = T_{c(s)}$ is a trace-preserving unital completely positive map satisfying $\Phi(b_i) = a_i$ for all $1 \leq i \leq 2$, but there is no locally mixed unitary unital completely positive map with this property.
\end{exam}

For any pair $b = (b_1, b_2) \in (\M_n)^2_{\sa}$, it is elementary to verify that $(\tr_n(b_1)1_n, \tr_n(b_2)1_n) \prec (b_1, b_2)$.  Thus non-commuting pairs can majorize commuting pairs.  The following provides an example of a commuting pair that majorizes a non-commuting pair.

\begin{exam}
Define $a,b \in \M_4$ by
\[
a = \frac{1}{\sqrt{3}}\begin{bmatrix}
\sqrt{3} i & 1 & 1 & i \\
1 & \sqrt{3} i & 1 & -i \\
1 & -1 & \sqrt{3} i & i \\
-i & i & -i & \sqrt{3} i
\end{bmatrix}, \quad b = \begin{bmatrix}
i & 1 & 1 & -i \\
1 & i & -i & 1 \\
1 & -i & i & 1 \\
-i & 1 & 1 & i
\end{bmatrix}.
\]
Define $\Phi : \M_4 \to \M_4$ by $\Phi = T_{c(s)}$ for $s = \frac{1}{3}$, where $T_{c(s)}$ is the Schur multiplier from Example \ref{exam:no-mixed-unitaries-for-2}. Then $\Phi$ is a trace-preserving unital completely positive map and $\Phi(b) = a$. Note that $a$ is not normal but $b$ is normal. Thus $a = (\Re(a), \Im(a))$ is a non-commuting pair of self-adjoint matrices and $b = (\Re(b), \Im(b))$ is a commuting pair of self-adjoint matrices, yet $a \prec b$ by Corollary \ref{cor:bistochastic}. The existence of such an example is not surprising since, for example, the convex hull of the joint unitary orbit of a pair of commuting operators must be a subset of pairs of self-adjoint operators, but there is no reason to expect all such convex combinations  still commute.
\end{exam}



\begin{thebibliography}{99}


%
%

%

%
%
%
%
%
%

%
%
%
%

%
%



\bib{DK2021}{article}{
    author={Davidson, K.R.},
    author={Kennedy, M.},
    title={Choquet order and hyperrigidity for function systems},
    journal={Adv. Math.},
    volume={385},
    year={2021},
    pages={107774}
}


\bib{DK2024}{article}{
    author={Davidson, K.R.},
    author={Kennedy, M.},
    title={Noncommutative Choquet theory: A Survey},
    eprint={arXiv:2412.09455},
    year={2024},
    pages={27}
}


\bib{DK2019}{article}{
    author={Davidson, K.R.},
    author={Kennedy, M.},
    title={Noncommutative choquet theory},
	journal={Mem. Amer. Math. Soc.},
    year={2025},
    volume={316},
    number={1608},
    pages={83 pages}
}

%
%

\bib{GJBDM2018}{article}{
    author={Gour, G.},
    author={Jennings, D.},
    author={Buscemi, F.},
    author={Duan, R.},
    author={Marvian, I.},
    title={Quantum majorization and a complete set of entropic conditions for quantum thermodynamics},
    journal={Nature communications},
    volume={9},
    number={1},
    year={2018},
    pages={5352}
}

\bib{HLP1952}{book}{
  title={Inequalities},
  author={Hardy, G.H.},
  author={Littlewood, J.E.},
  author={P{\'o}lya, G.},
  year={1952},
  publisher={Cambridge university press}
}

\bib{HM2011}{article}{
    author={Haargerup, U.},
    author={Musat, M.},
    title={Factorization and dilation problems for completely positive maps on von Neumann algebras},
    journal={Commun. Math. Phys.},
    volume={303},
    number={2},
    year={2011},
    pages={555-594}
} 
 

\bib{HKM2017}{article}{
    author={Helton, J.W.},
    author={Klep, I.},
    author={McCullough, S.},
    title={The tracial Hahn-Banach theorem, polar duals, matrix convex sets, and projections of free spectrahedra},
    journal={J. Eur. Math. Soc.},
    volume={19},
    number={6},
    year={2017},
    pages={1845-1897}
} 
%
%
%
%
%
%
%
%

\bib{K1951}{article}{
  author  = {Kadison, Richard V.},
  title   = {Order properties of bounded self-adjoint operators},
  journal = {Proceedings of the National Academy of Sciences},
  volume  = {37},
  number  = {4},
  pages   = {201--203},
  year    = {1951},
  publisher = {National Acad Sciences}
}

%
%
%


\bib{MMS2005}{article}{
    author={Mart{\'\i}nez Per{\'\i}a, F.},
    author={Massey, P.},
    author={Silvestre, L.},
    title={Weak matrix majorization},
    journal={Linear Algebra Appl},
    volume={403},
    number={3},
    date={2005},
    pages={343-368}
}

%
%
%
%
%
%
%
%
%
%
%

\bib{T1986}{article}{
  author  = {Tregub, S. L.},
  title   = {Bistochastic operators on finite-dimensional von {N}eumann algebras},
  journal = {Izv. Vyssh. Uchebn. Zaved. Mat.},
  number  = {3},
  pages   = {75--77},
  year    = {1986},
  note    = {English translation: \emph{Soviet Math. (Iz. VUZ)} \textbf{30}(3): 105--108, 1986}
}

\bib{U1970}{article}{
  author  = {Uhlmann, Armin},
  title   = {S{\"a}tze {\"u}ber die {T}ransformation von {D}ichtematrizen},
  journal = {Wiss. Z. Karl-Marx-Univ. Leipzig Math.-Natur. Reihe},
  volume  = {19},
  pages   = {571--575},
  year    = {1970}
}

%
%



\end{thebibliography}
\end{document}